\documentclass[a4paper,11pt]{article}
\usepackage{makeidx}
\usepackage{graphicx}
\usepackage[english,activeacute]{babel}
\usepackage[latin1]{inputenc}
\usepackage{amsmath}
\usepackage{amsfonts}
\usepackage{amssymb}
\usepackage{amsthm}
\usepackage{latexsym}
\usepackage{colortbl}
\usepackage{multicol}
\usepackage{hyperref}
\usepackage{xcolor}
\usepackage{orcidlink}
\usepackage{mathrsfs}
\usepackage{tikz}
\usepackage{pgfplots}
\pgfplotsset{compat=1.18}
\usepackage{mathtools}

\newtheorem{theorem}{Theorem}[section]
\newtheorem{proposition}[theorem]{Proposition}
\newtheorem{lemma}[theorem]{Lemma}

\newtheorem{example}[theorem]{Example}

\newcommand{\N}{\mathbb{N}}

\newcommand{\R}{\mathbb{R}}
\newcommand{\C}{\mathbb{C}}
\newcommand{\PP}{\mathbb{P}}

\newcommand{\Charlier}[1]{C^{(\mu)}_{#1}}      
\newcommand{\Meixner}[1]{M^{(\gamma,\mu)}_{#1}}
\newcommand{\Sob}[1]{\mathbb{S}_{#1}}          
\newcommand{\dif}{\Delta}                       
\newcommand{\nab}{\nabla}                       
\newcommand{\norm}[1]{\lVert #1\rVert}
\newcommand{\Kn}{K}                             
\newcommand{\ff}[2]{\left[#1\right]_{#2}}       
\newcommand{\inner}[2]{\left\langle #1,#2\right\rangle}
\newcommand{\Aop}{\mathcal{A}}
\newcommand{\Bop}{\mathcal{B}}

\newcommand{\Nf}{\mathcal{N}}

\providecommand{\dst}{\displaystyle}

\begin{document}

\title{On generating functions and Mehler--Heine formulas for discrete Charlier and Meixner Sobolev-type orthogonal polynomials}
\author{{Anier Soria-Lorente %
\orcidlink{0000-0003-3488-3094}}$^{1}$,{Junior Michel \orcidlink{0009-0005-8450-2757}}$^{1}$, \\
\\
$^{1}$Department of Quantitative Methods, Loyola University,\\ Avda. de
las Universidades, 2. Dos Hermanas, Seville,\\ 41704, Andalusia, Spain.\\
asoria@uloyola.es, jmichel@uloyola.es
}

\maketitle
\begin{abstract}
Generating functions are among the most important analytical tools in the theory of orthogonal polynomials, providing a unified framework for deriving structural identities, asymptotic expansions, and zero distributions. However, despite the extensive development of discrete Sobolev orthogonal polynomials, no general generating-function theory has been available for the Sobolev-type Charlier and Meixner families associated with arbitrary-order forward differences $j\geq 1$ and an exterior mass point $\alpha<0$. In this paper, we develop the first unified generating-function framework for these families. Starting from explicit connection formulas, we derive generating functions for the Sobolev-type polynomials and their iterated forward differences, which serve as the main analytical tool for establishing new Mehler--Heine formulas. The resulting asymptotic analysis shows that an exterior Sobolev perturbation generates exactly one exceptional zero converging to the mass point, while the remaining zeros preserve the classical asymptotic distribution. Moreover, the limiting Mehler--Heine functions are independent of both the Sobolev mass parameter and the order of the forward difference operator, revealing a universality phenomenon for higher-order discrete Sobolev perturbations. These results considerably extend the analytical theory of discrete Sobolev orthogonal polynomials and establish a direct connection between generating functions, Mehler--Heine asymptotics, and the asymptotic distribution of the zeros for the discrete Sobolev-type Charlier and Meixner families.
\end{abstract}

\vspace{0.3cm}


{\footnotesize Corresponding author: Anier Soria-Lorente}



\section{Introduction}

Orthogonal polynomials constitute one of the fundamental pillars of
approximation theory and the theory of special functions. Their rich
algebraic and analytic structure has made them indispensable in a wide
variety of mathematical disciplines, including numerical analysis,
probability theory, mathematical physics, random matrix theory, and
spectral methods. Classical families possess a remarkable collection of
structural properties, such as three-term recurrence relations,
difference or differential equations, Christoffel--Darboux kernels,
Rodrigues-type formulas, generating functions, and asymptotic
expansions, all of which are intimately connected and provide
complementary descriptions of the same polynomial system
\cite{Ismail05,KLS2010,nikiforov1991classical}.

Among these analytical tools, generating functions occupy a particularly
distinguished position. Far from being merely compact representations of
polynomial sequences, they provide a unified framework from which many
of the fundamental properties of orthogonal polynomials can be derived.
Generating functions encode the entire polynomial family into a single
analytic object, allowing one to obtain recurrence relations,
connection formulas, lowering and raising operators, explicit
coefficients, combinatorial identities, integral representations, and,
perhaps most importantly, asymptotic information. For many classical
families, they constitute the natural starting point for studying both
their algebraic structure and their asymptotic behaviour
\cite{Ismail05,KLS2010}.

Within the Askey scheme of hypergeometric orthogonal polynomials, the
Charlier and Meixner families occupy a central position as the canonical
discrete analogues of the Hermite and Laguerre families. Besides their
importance in the general theory of special functions, they are closely
related to the Poisson and negative binomial probability
distributions, respectively, and arise naturally in stochastic
processes, queueing theory, birth--death models, and discrete
approximation problems. Their classical theory is exceptionally rich,
including explicit generating functions, hypergeometric
representations, ladder operators, second-order difference equations,
reproducing kernels, and Mehler--Heine asymptotic formulas
\cite{Ismail05,KLS2010,nikiforov1991classical}. These properties make
them natural candidates for investigating how classical analytical
structures survive under nonstandard perturbations.

One of the most important generalizations of classical orthogonality is
provided by Sobolev-type orthogonal polynomials, where the underlying
inner product is modified by incorporating derivatives or, in the
discrete setting, finite differences evaluated at prescribed points.
Such perturbations significantly enrich the algebraic and asymptotic
properties of the corresponding polynomial families while preserving
many of the structural features that make orthogonal polynomials useful
in analysis and applications. During the last three decades, Sobolev
orthogonality has developed into one of the most active areas within the
constructive theory of orthogonal polynomials, motivated both by its
intrinsic mathematical interest and by its applications to spectral
methods, approximation theory, interpolation, and numerical analysis
\cite{maxu,mapepi,khol}.

The discrete Sobolev framework is particularly attractive because it
extends many of the classical concepts of the continuous theory while
introducing genuinely new phenomena. Replacing derivatives by forward or
backward difference operators leads to families whose analytical
behaviour is considerably more intricate than in the classical setting.
Connection formulas become substantially more involved, recurrence
relations increase in order, and new holonomic difference equations
appear. These structural changes also raise natural questions regarding
the analytical tools that play a fundamental role in the classical
theory. Among them, generating functions and Mehler--Heine asymptotic
formulas stand out because they provide a global description of the
polynomial sequence and its limiting behaviour.

Although remarkable progress has been achieved in the algebraic theory
of discrete Sobolev orthogonal polynomials, the analytical theory
remains considerably less developed. In particular, while connection
formulas, recurrence relations, ladder operators, and difference
equations have been extensively investigated, comparatively little is
known about generating functions in the general discrete Sobolev
setting. This absence is particularly significant because generating
functions are often the analytical bridge connecting the algebraic
properties of orthogonal polynomials with their asymptotic behaviour.
Without them, many classical techniques become unavailable, making the
development of a unified asymptotic theory considerably more difficult.

The objective of the present paper is precisely to address this
fundamental gap for the discrete Sobolev-type Charlier and Meixner
families associated with arbitrary-order forward differences evaluated
at an exterior mass point. As will be shown throughout the paper, the
construction of the corresponding generating functions is far from a
routine extension of the classical theory. The Sobolev perturbation
introduces a degree-dependent reproducing-kernel denominator that
completely destroys the recursive mechanism underlying the classical
generating functions. Overcoming this obstacle requires a different
analytical approach and ultimately leads to a unified framework from
which new asymptotic results naturally emerge.

In particular, the generating functions developed here become the
central analytical tool for establishing new Mehler--Heine formulas and
for describing the asymptotic distribution of the zeros of the
Sobolev-type families. As a consequence, the present work not only fills
an important gap in the existing literature but also reveals new
asymptotic phenomena that do not appear in the previously studied
boundary cases.

\subsection{Related Work}

The theory of Sobolev orthogonal polynomials originated from the study
of orthogonality with respect to inner products involving derivatives
evaluated together with classical measures. Since the pioneering work of
Marcell\'an and collaborators~\cite{MR1990,almarero}, this subject has
evolved into one of the most active branches of the constructive theory
of orthogonal polynomials. Comprehensive surveys describing its
development and principal achievements can be found in
Marcell\'an and Xu~\cite{maxu}, Marcell\'an, P\'erez and
Pi\~nar~\cite{mapepi}, and Koekoek~\cite{khol}.

The discrete counterpart of Sobolev orthogonality was introduced by
Bavinck~\cite{B1995,B1995gen,B1996}, who replaced derivatives by finite
difference operators and established the first analytical properties of
the resulting polynomial families. Besides introducing the corresponding
Sobolev-type Charlier polynomials, Bavinck constructed an infinite-order
difference operator having these polynomials as eigenfunctions,
demonstrating that many structural features of the classical theory
persist under discrete Sobolev perturbations while, at the same time,
revealing the appearance of genuinely new phenomena.

Subsequent research considerably expanded the algebraic theory of
discrete Sobolev orthogonal polynomials. The Meixner--Sobolev family was
systematically investigated by \'Area, Godoy and
Marcell\'an~\cite{agm}, who obtained its three-term recurrence relation,
ladder operators, and hypergeometric representation. Later,
\'Area, Godoy, Marcell\'an and Moreno-Balc\'azar~\cite{agmmb}
established the corresponding ratio and Plancherel--Rotach asymptotic
formulas, providing a detailed description of their large-degree
behaviour.

More recently, attention has shifted towards higher-order Sobolev
perturbations. Huertas and Soria-Lorente~\cite{HS2019} derived several
new structural properties for discrete Sobolev-type Charlier
polynomials, whereas Costas-Santos, Soria-Lorente and
Vilaire~\cite{Costas-2022} developed a unified algebraic framework for
Sobolev-type Meixner polynomials associated with arbitrary-order forward
and backward differences. Their work established explicit connection
formulas, hypergeometric representations, ladder operators,
$(2j+3)$-term recurrence relations, and second-order holonomic
difference equations. It also obtained a Mehler--Heine formula in the
special boundary case $\alpha=0$, where the Sobolev mass is located at
the endpoint of the support and the limiting behaviour remains
essentially classical.

Parallel developments have confirmed that holonomic second-order
difference equations constitute a common structural feature of discrete
Sobolev families. Related equations have been obtained by
Rebocho~\cite{R22} and, in the Hahn setting, by Filipuk,
Ma\~nas-Ma\~nas and Moreno-Balc\'azar~\cite{gamamo}. Likewise,
Dominici and Moreno-Balc\'azar~\cite{domo} investigated the asymptotic
behaviour of a closely related Charlier--Sobolev family, further
demonstrating the richness of the asymptotic theory associated with
discrete Sobolev perturbations.

Collectively, these contributions have established a comprehensive
algebraic theory for discrete Sobolev-type orthogonal polynomials.
Connection formulas, reproducing kernels, ladder operators, recurrence
relations, hypergeometric representations, and holonomic difference
equations are now well understood for several important families.
Nevertheless, one of the most fundamental analytical ingredients of the
classical theory has remained surprisingly underdeveloped, namely the
theory of generating functions.

For the classical Charlier and Meixner polynomials, generating functions
provide considerably more than compact representations of the polynomial
sequence. They encode the complete family into a single analytic object
from which numerous structural properties follow naturally, including
connection formulas, lowering operators, explicit coefficients,
combinatorial identities, and asymptotic expansions
\cite{Ismail05,KLS2010}. In many situations, they constitute the
starting point for deriving Mehler--Heine formulas and for analysing the
asymptotic behaviour of the zeros.

In contrast, generating functions for discrete Sobolev-type families are
remarkably scarce. A generating function for a particular nonstandard
Meixner--Sobolev family involving first-order differences was obtained
by Moreno-Balc\'azar, P\'erez and Pi\~nar~\cite{mbtpmp}. Later,
Moreno-Balc\'azar~\cite{mbDMS} established a Mehler--Heine formula and
studied the asymptotic behaviour of the zeros for the
$\Delta$-Meixner--Sobolev polynomials corresponding to a first-order
difference and a Sobolev mass located at the endpoint of the support.
Interestingly, in that situation the Sobolev perturbation modifies the
classical Mehler--Heine limit only through a multiplicative constant,
and consequently the asymptotic zero distribution remains unchanged.

The situation is fundamentally different when the Sobolev mass is placed
outside the support of the orthogonality measure. In this setting, the
connection formula contains the degree-dependent denominator
\begin{equation*}
    \delta_n=1+\lambda\,\Kn^{(j,j)}_{n-1}(\alpha,\alpha),
\end{equation*}
whose dependence on the polynomial degree destroys the recursive
mechanism that underlies the classical generating functions. As a
result, the analytical techniques available in the classical theory no
longer apply directly, and the construction of explicit generating
functions becomes a substantially more delicate problem.

To the best of our knowledge, no general generating-function theory has
previously been developed for discrete Sobolev-type Charlier and
Meixner orthogonal polynomials associated with arbitrary-order forward
differences evaluated at an exterior mass point. Consequently, the
corresponding Mehler--Heine asymptotics and the asymptotic behaviour of
their zeros have also remained essentially unexplored in this general
framework.

\subsection{Our Contributions}

The present paper closes the gap described above by developing, to the
best of our knowledge, the first unified generating-function framework
for discrete Sobolev-type Charlier and Meixner orthogonal polynomials
associated with arbitrary-order forward differences evaluated at an
exterior mass point. Unlike the classical setting, where generating
functions follow naturally from the recurrence structure of the
polynomial sequence, the Sobolev perturbation introduces the
degree-dependent factor $\delta_n$
which completely destroys the recursive mechanism underlying the
classical generating functions. Consequently, the construction of
explicit generating functions requires a different analytical approach
combining connection formulas, reproducing kernels, asymptotic analysis,
and careful estimates for the Sobolev correction terms.

Our first main contribution consists in deriving explicit generating
functions for both discrete Sobolev-type families. More precisely, we
obtain closed generating-function representations for the Sobolev
polynomials themselves as well as for their iterated forward
differences. These representations extend the classical generating
functions of the Charlier and Meixner families to a considerably more
general Sobolev framework and provide a unified analytic description
valid for arbitrary orders of the forward difference operator.

Beyond their intrinsic interest, these generating functions become the
central analytical tool of the paper. They allow us to develop a unified
asymptotic analysis leading to new Mehler--Heine formulas for both
Sobolev-type families. To the best of our knowledge, these are the first
Mehler--Heine formulas established for arbitrary-order discrete
Sobolev-type Charlier and Meixner polynomials with an exterior mass
point. Unlike the previously investigated boundary case
$\alpha=0$, where the Sobolev perturbation modifies the classical limit
only through a multiplicative constant, the present setting exhibits a
qualitatively different asymptotic behaviour.

A second major contribution of the paper is the complete description of
the asymptotic distribution of the zeros. Using the new Mehler--Heine
formulas, we prove that the exterior Sobolev perturbation generates a
single exceptional zero converging to the mass point
$\alpha$, whereas all remaining zeros converge to the zeros of the
corresponding classical limiting function. This behaviour reveals how
the Sobolev mass influences the global asymptotic structure of the
polynomial sequence while preserving the classical distribution of the
remaining zeros.

Perhaps the most remarkable outcome of our analysis is the emergence of
a universality phenomenon. Although the finite-degree Sobolev-type
polynomials depend explicitly on both the Sobolev parameter
$\lambda$ and the order $j$ of the forward difference operator, the
limiting Mehler--Heine functions turn out to be completely independent
of both quantities. Consequently, the asymptotic behaviour is governed
only by the underlying classical family, despite the presence of
higher-order Sobolev perturbations. This result highlights an unexpected
stability of the asymptotic regime and provides new insight into the
interaction between discrete Sobolev perturbations and classical
orthogonal polynomial theory.

From a broader perspective, the results presented here considerably
extend the current analytical theory of discrete Sobolev orthogonal
polynomials. Besides filling a significant gap in the existing
literature, they establish generating functions as an effective
analytical framework for studying discrete Sobolev families and open the
door to further investigations on asymptotic analysis, zero
distribution, connection problems, and discrete Sobolev perturbations
for other families within the Askey scheme.

The paper is organized as follows. Section~\ref{sec:prelim} introduces
the classical Charlier and Meixner families together with the discrete
Sobolev inner product and the corresponding connection formulas.
Section~\ref{sec:formalGF} develops the generating-function theory for
the Sobolev-type polynomials and their iterated forward differences.
Section~\ref{sec:mehlerheine} establishes the new Mehler--Heine formulas
and derives the asymptotic distribution of the zeros. Finally,
Section~\ref{sec:conclu} summarizes the main conclusions and discusses
several directions for future research.

\section{Mathematical Background}
\label{sec:prelim}

In this section we recall the main definitions and results on the
Charlier and Meixner families of orthogonal polynomials, and we introduce the
discrete Sobolev-type Charlier and Meixner polynomials that are the object of this
paper. Following the unified style of~\cite{Costas-2022,KLS2010}, we present the
results in a compact form, valid for both families upon particularising the two
columns of Table~\ref{tab:families}. Throughout, $\dif$ and $\nab$ denote
the forward and backward difference operators
\begin{equation}\label{eq:diffops}
\dif f(x)=f(x+1)-f(x),\quad \nab f(x)=f(x)-f(x-1),
\end{equation}
applied recursively by
\begin{equation}\label{forWop}
    \dif^{n}f(x)=\dif[\dif^{n-1}f(x)],
\end{equation}
and a subscript
$\dif_x$ or $\dif_y$ indicates the variable acted upon. The forward
operator satisfies the product rule
\begin{equation}\label{eq:prodrule}
\dif[f(x)g(x)]=f(x)\dif g(x)+g(x+1)\dif f(x).
\end{equation}
Throughout this paper, $[\cdot]_k$ denotes the falling factorial of order $k$, defined by
\begin{equation*}
    [x]_k=\begin{cases}
        1, & k=0,\\\\
        \dst\prod_{j=0}^{k-1}(x-j), & k\geq 1.
    \end{cases}
\end{equation*}
Moreover,
\begin{equation*}
    (x)_k=(-1)^k[-x]_k=\frac{\Gamma(x+k)}{\Gamma(x)},
\end{equation*}
denotes the Pochhammer symbol of order $k$, and $\Gamma$ stands for the Gamma function. Let $\PP$ denote the linear space of polynomials with real coefficients.

Fix one of the two columns of Table~\ref{tab:families}, that is, either
the monic Charlier polynomials $\Charlier{n}(x)$ with $\mu>0$ or the
monic Meixner polynomials $\Meixner{n}(x)$ with $\gamma>0$ and
$0<\mu<1$. We denote by $P_n$ the corresponding generic monic family and
refer to~\cite{Ismail05,KLS2010,nikiforov1991classical} for the
background.

\begin{table}[h]
\centering
\caption{Monic Charlier $\Charlier{n}(x)$ and Meixner
$\Meixner{n}(x)$ polynomials and their data.}
\label{tab:families}
\renewcommand{\arraystretch}{2.0}
\begin{tabular}{|c|c|c|}
\hline
$P_{n}(x)$ & $\Charlier{n}(x)$, $\mu>0$ &
$\Meixner{n}(x)$, $\gamma>0$, $0<\mu<1$\\ \hline\hline
$\rho(x)$ & $\dfrac{e^{-\mu}\mu^{x}}{\Gamma(x+1)}$ &
$\dfrac{\mu^{x}\Gamma(\gamma+x)}{\Gamma(\gamma)\Gamma(x+1)}$\\ \hline
$\norm{P_{n}}^{2}$ & $n!\,\mu^{n}$ &
$\dfrac{n!\,(\gamma)_{n}\,\mu^{n}}{(1-\mu)^{\gamma+2n}}$\\ \hline
$\alpha_{n}$ & $n+\mu$ & $\dfrac{n(1+\mu)+\mu\gamma}{1-\mu}$\\ \hline
$\beta_{n}$ & $n\mu$ & $\dfrac{n\mu(n-1+\gamma)}{(\mu-1)^{2}}$\\ \hline
\end{tabular}
\end{table}

\begin{proposition}\label{prop:classical}
	Let $\{P_{n}\}_{n\geq 0}$ be the sequence of monic classical discrete polynomials of degree $n$. The following statements hold.
	
	\begin{enumerate}
		\item The sequence $\{P_{n}\}_{n\geq 0}$ consists of monic polynomials orthogonal with respect to the inner product defined on the space of polynomials $\PP$, 
		\begin{equation*}
		\left\langle f,g\right\rangle =\sum_{x\geq 0}f(x)g(x)\rho(x),
		\end{equation*}
		where $\rho(x)$ is defined by the corresponding column of the second row of Table \ref{tab:families}.

		\item Squared norm. For every $n\in\mathbb{N}$, the quantity $\|P_n\|^2=\left\langle P_n,P_n\right\rangle$ is given in Table~\ref{tab:families}.

        \item The three-term recurrence relation
        \begin{equation}
        xP_n(x)=P_{n+1}(x)+\alpha_nP_n(x)+\beta_nP_{n-1}(x),
        \quad n\ge 0,
        \label{eq:ttrr}
        \end{equation}
        holds, with initial conditions $P_{-1}\equiv0$ and $P_0\equiv1$, where
        the recurrence coefficients $\alpha_n$ and $\beta_n$ are given in the
        corresponding columns of Table~\ref{tab:families}.
		
		\item Generating functions. Throughout this paper the classical
		generating function of a monic family $\{P_n\}_{n\ge0}$ is taken in the
		\emph{monic-normalised} form
		\begin{equation}\label{eq:GPmonic}
			G_P(x,t)=\sum_{n=0}^{\infty}\frac{\psi_n}{n!}\,P_n(x)\,t^n,
		\end{equation}
		where the normalising sequence $\psi_n$ is fixed by the requirement that
		$G_P$ coincide with the standard closed form, namely
		$\psi_n=(-\mu)^{-n}$ in the Charlier case and
		$\psi_n=(\mu-1)^{n}\mu^{-n}$ in the Meixner case.
		\begin{itemize}
			\item Charlier case ($|t|<\mu$):
			\begin{equation}\label{eq:GCmonic}
				G_C(x,t)=e^t\left( 1-\frac{t}{\mu}\right)^x= \sum_{n=0}^{\infty}\frac{(-\mu)^{-n}}{n!}\Charlier{n}(x)t^n,\quad\mu>0.
			\end{equation}
			
			\item Meixner case ($|t|<\mu$):
			\begin{equation}\label{eq:GMmonic}
				G_M(x,t)=\left( 1-\frac{t}{\mu}\right)^x(1-t)^{-x-\gamma}= \sum_{n=0}^{\infty}\frac{(\mu-1)^{n}}{n!\mu^n}\Meixner{n}(x)t^n,
			\end{equation}
		\end{itemize}
		where, $\gamma>0$, $0<\mu<1$. In the Meixner case the two singularities of the closed form sit at
		$t=\mu$ and $t=1$; since $0<\mu<1$, the radius of convergence is the
		smaller value $\mu$, so both families converge on the common disc
		$|t|<\mu$.

        \item Forward shift operator. For every $n\ge1$ and $0\le k\le n$, the
        $k$-th forward difference of a monic classical discrete family lowers
        the degree by $k$ according to        \begin{equation}\label{eq:fwdgeneral}
        \dif^k P_n(x)=[n]_k\,P_{n-k}(x),\quad 0\leq k\leq n.
        \end{equation}
        In each
        particular case this reads as follows. For the monic Charlier
        polynomials $\Charlier{n}$ with $\mu>0$,
        \begin{equation}\label{eq:fwdCharlier}
        \dif^k\Charlier{n}(x)=[n]_k\,\Charlier{n-k}(x),\quad 0\leq k\leq n,
        \end{equation}
        whereas for the monic Meixner polynomials $\Meixner{n}$ with
        $\gamma>0$, $0<\mu<1$, the lowered family carries a shifted parameter,
        \begin{equation}\label{eq:fwdMeixner}
        \dif^k\Meixner{n}(x)=[n]_k\,M_{n-k}^{\gamma+k,\mu}(x),\quad 0\leq k\leq n.
        \end{equation}

		\item Mehler--Heine asymptotics. In the monic normalisation of
        Table~\ref{tab:families}, Dominici's formulas~\cite{Dominici2015}
        read, for the monic Charlier polynomials $\Charlier{n}$ with $\mu>0$,
        \begin{equation}\label{eq:mhCclassical}
        \lim_{n\to\infty}\frac{(-1)^{n}}{\Gamma(n-z)}\,\Charlier{n}(z)
        =\varphi_C(z),\quad \varphi_C(z)=\frac{e^{\mu}}{\Gamma(-z)},
        \end{equation}
        and, for the monic Meixner polynomials $\Meixner{n}$ with $\gamma>0$,
        $0<\mu<1$,
        \begin{equation}\label{eq:mhMclassical}
        \lim_{n\to\infty}\frac{(\mu-1)^{n}}{\Gamma(n-z)}\,\Meixner{n}(z)
        =\varphi_M(z),\quad
        \varphi_M(z)=\frac{1}{(1-\mu)^{\gamma+z}\,\Gamma(-z)},
        \end{equation}
        both uniformly on compact subsets of $\C$. These are the monic forms
        of~\cite[Eqs.~(23),(35)]{Dominici2015}. Writing $\varphi$ for either
        limit function and $\kappa_n$ for the corresponding sign/scale, both
        relations take the common form
        \begin{equation}\label{eq:mhcommon}
        \lim_{n\to\infty}\frac{1}{\kappa_n\,\Gamma(n-z)}\,P_n(z)=\varphi(z),
        \quad
        \kappa_n=
        \begin{cases}
        (-1)^{n}, & \text{Charlier},\\[2pt]
        (\mu-1)^{-n}, & \text{Meixner}.
        \end{cases}
        \end{equation}
        Both limit functions are entire multiples of $1/\Gamma(-z)$; hence
        their only zeros are the simple zeros of $1/\Gamma(-z)$ at
        $z=0,1,2,\dots$.
    \end{enumerate}
\end{proposition}

Concerning Charlier and Meixner families of orthogonal polynomials, a Christoffel-Darboux formula is also available.

\begin{proposition}[Christoffel-Darboux formula]\label{pro:CDF}
	Let $\{P_{n}\}_{n\geq 0}$ be one of the two sequences of monic classical discrete polynomials of degree $n$ determined in Table~\ref{tab:families}. Let $K_n$ denote the $n$-th reproducing kernel, defined by 
	\begin{equation}\label{eq:kerneldef}
	K_{n}(x,y)=\sum_{k=0}^{n}\frac{P_{k}(x)P_{k}(y)}{||P_{k}||^{2}}.
	\end{equation}
	Then, for all $n\in \mathbb{N}$, it holds that 
	\begin{equation}
	K_{n}(x,y)=\frac{P_{n+1}(x) P_{n}(y) -P_{n+1}(y)P_{n}(x) }{\left(
		x-y\right) ||P_{n}||^{2}}.  \label{CDarb}
	\end{equation}
\end{proposition}
Let us also fix the following notation on the iterative application of the forward difference operator $K_{n}(x,y)$ with respect to each variable: for all $(i,j)\in\N^2$ we write
\begin{equation}
K_{n}^{(i,j)}(x,y)=\Delta_{x}^{i}\left( \Delta_{y}^{j}K
_{n}\left( x,y\right) \right) =\sum_{k=0}^{n}\frac{%
	\Delta^{i}P_{k}(x)\Delta^{j}P_{k}(y)}{\left\Vert
	P_{k}\right\Vert ^{2}}.  \label{eq:Kij}
\end{equation}
The following result can be proved by adapting the argument of \cite[Proposition 3]%
{Costas-2022} with only minor modifications. We therefore give only a brief proof.
\begin{proposition}\label{S1-LemmaKernel0j}
Let $\{P_{n}\}_{n\geq 0}$ be as above.The following statement holds for
	every $n\in \mathbb{N}$, 
	\begin{equation}
	K_{n-1}^{\left( 0,j\right) }\left( x,y\right) = {\mathcal{A}}_n^{(j)}(x,y)P_{n}\left( x\right)+{\mathcal{B}}_n^{(j)}(x,y)P_{n-1}\left( x\right),  \label{eq:K0j}
	\end{equation}
	with
	\begin{equation}\label{eq:Aj}
	{\mathcal{A}}_n^{(j)}(x,y) = \frac{j!}{\left\Vert P_{n-1}\right\Vert
		^{2}[x-y]_{j+1}}\sum_{k=0}^j\frac{\Delta^{k}P_{n-1}\left(
		y\right) }{k!}[x-y]_{k},
	\end{equation}
	and 
	\begin{equation}\label{eq:Bj}
	{\mathcal{B}}_n^{(j)}(x,y) = -\frac{j!}{\left\Vert P_{n-1}\right\Vert
		^{2}[x-y]_{j+1}}\sum_{k=0}^j\frac{\Delta^{k}P_{n}\left(
		y\right) }{k!}[x-y]_{k}.
	\end{equation}
\end{proposition}

We now introduce the discrete Sobolev-type polynomials of order $j$,
$\{\Sob{n}(x)\}_{n\ge0}$, orthogonal with respect to the Sobolev-type
inner product
\begin{equation}\label{eq:innerprod}
\inner{f}{g}_{\lambda}
=\sum_{x\ge0}f(x)g(x)\rho(x)+\lambda\,\dif^{j}f(\alpha)\,\dif^{j}g(\alpha),
\end{equation}
where $\alpha\in\R_{-}$, $\lambda\in\R^{+}$ and $j\in\N$ with $j\ge1$ are
fixed, and
$\rho$ is the weight of the chosen column of Table~\ref{tab:families}. As before, $P_n$ denotes the
underlying classical monic family and all notation of
Section~\ref{sec:prelim} is retained.

The link between the Sobolev-type polynomials $\Sob{n}$ and the classical
polynomials $P_n$ is the following connection formula, which underlies
all subsequent constructions.

\begin{proposition}\label{prop:connection}
Let $\{\Sob{n}\}_{n\ge0}$ be the monic Sobolev-type family for the inner
product~\eqref{eq:innerprod}. Then, for $n\ge1$,
\begin{equation}\label{eq:ConxF1}
\Sob{n}(x)=P_{n}(x)
-\lambda\,\frac{\dif^{j}P_{n}(\alpha)}
{1+\lambda\,\Kn_{n-1}^{(j,j)}(\alpha,\alpha)}\,
\Kn_{n-1}^{(0,j)}(x,\alpha).
\end{equation}
\end{proposition}

\begin{proof}
The argument of~\cite[Section~2]{MR1990} applies verbatim: expanding
$\Sob{n}$ in the basis $\{P_k\}_{k\le n}$ and imposing orthogonality with
respect to~\eqref{eq:innerprod}, all Fourier coefficients but the one
carrying the mass term vanish, and the reproducing property of the
kernel~\eqref{eq:kerneldef} yields~\eqref{eq:ConxF1}.
\end{proof}

For convenience, we introduce the abbreviated notation
\begin{equation}\label{eq:SobCon}
    \Sob{n}(x)=P_{n}(x)-\lambda\,\Omega_{n}\,\Kn_{n-1}^{(0,j)}(x,\alpha),
\end{equation}
where
\begin{equation}\label{eq:Omega}
\Omega_{n}=\frac{\dif^{j}P_{n}(\alpha)}
{1+\lambda\,\Kn_{n-1}^{(j,j)}(\alpha,\alpha)}=\frac{\dif^{j}P_{n}(\alpha)}
{\delta_n}.
\end{equation}


\section{The Sobolev generating function}\label{sec:formalGF}

This section establishes the generating function of the Sobolev-type
polynomials and settles its analytic status. The guiding idea is simple to
state. The connection formula~\eqref{eq:Omega} writes each $\Sob{n}$ as the
classical $P_n$ minus a rank-one Sobolev correction; summing against
$t^{n}$ should therefore present the Sobolev generating function as the
classical one minus a correction series. The Sobolev-type generating function is defined with the
\emph{same} normalising sequence,
\begin{equation}\label{eq:GFdef}
G_S(x;t)=\sum_{n\ge0}\frac{\psi_n}{n!}\,\Sob{n}(x)\,t^{n}.
\end{equation}
Using a common $\psi_n/n!$ for both families places $G_S$ and $G_P$ on the
same footing, so that the connection formula~\eqref{eq:Omega} transcribes
into a clean identity between the two series and the closed forms
\eqref{eq:GCmonic}--\eqref{eq:GMmonic} are available on the common disc
$|t|<\mu$.

\begin{lemma}\label{lem:cauchy}
Let $\{P_n\}$ be the monic Charlier family, $\mu>0$, and
$\psi_n=(-\mu)^{-n}$ the normalising sequence of~\eqref{eq:GCmonic}. For
every $x\in\C$ and $0<r<\mu$,
\begin{equation}\label{eq:cauchybound}
\frac{|\psi_n|\,|P_n(x)|}{n!}=\frac{|P_n(x)|}{\mu^{n}\,n!}\le \frac{B_r(x)}{r^{\,n}},\quad
B_r(x)=e^{r}\Bigl(\tfrac{\mu}{\mu-r}\Bigr)^{\lvert\Re x\rvert}
e^{\,\frac{r}{\mu-r}\lvert\Im x\rvert},
\end{equation}
with $B_r(x)$ continuous in $x$ and independent of $n$. Equivalently,
\begin{equation*}
    \frac{|P_n(x)|}{n!}\le B_r(x)\left(\frac{\mu}{r}\right)^{n}.
\end{equation*}
\end{lemma}

\begin{proof}
Here $t$ denotes a complex variable. By the monic generating
function~\eqref{eq:GCmonic}, the numbers
\begin{equation*}
    \frac{\psi_n P_n(x)}{n!},
\end{equation*}
are the Taylor coefficients at $t=0$ of
\begin{equation*}
    F(t,x)=G_C(x,t)=e^t\left(1-\frac{t}{\mu}\right)^x,
\end{equation*}
regarded, for each fixed $x\in\C$, as a function of $t\in\C$. This function is
holomorphic in the disc $|t|<\mu$: the factor $e^{t}$ is entire, and
\begin{equation*}
    1-\frac{t}{\mu}\neq0,\quad |t|<\mu,
\end{equation*}
so the principal logarithm, and with it the
power
\begin{equation*}
    \left(1-\frac{t}{\mu}\right)^{x}=e^{x\operatorname{Log}(1-t/\mu)},
\end{equation*}
is analytic there.
The proof therefore reduces to Cauchy's coefficient estimate, for which it suffices
to bound $F(\cdot,x)$ uniformly on a circle $|t|=r$ with $0<r<\mu$.

Fix such an $r$. Cauchy's estimate on the circle $|t|=r$ yields
\begin{equation}\label{eq:cauchystep}
\frac{|\psi_n|\,|P_n(x)|}{n!}\le \frac{1}{r^{\,n}}\,\max_{|t|=r}\bigl|F(t,x)\bigr|,
\end{equation}
so it remains to estimate
\begin{equation*}
    |F(t,x)|=|e^{t}|\,\left|\left(1-\frac{t}{\mu}\right)^{x}\right|,
\end{equation*}
on $|t|=r$. The exponential
factor is immediate: since $\Re t\le|t|=r$, we have $|e^{t}|\le e^{r}$.

For the power factor, write $1-t/\mu=\rho\,e^{i\theta}$ in polar form, with
$\rho=|1-t/\mu|$ and $\theta=\arg(1-t/\mu)$. As $t$ traverses
$|t|=r$, the point $1-t/\mu$ traverses the circle $C$ of radius
$s=r/\mu$ centred at $1$, which lies in the right half-plane because
$s<1$. Two geometric facts about $C$ control $\rho$ and $\theta$:
\begin{equation}\label{eq:rhotheta}
1-s\le\rho\le 1+s,
\quad
|\theta|\le \arcsin s .
\end{equation}
The bounds on $\rho$ express that the distance from the origin to a point of $C$
lies between $1-s$ and $1+s$. The bound on $\theta$ follows because, among the
points of $C$, the argument attains its extreme values at the two points where a
line through the origin is tangent to $C$; for such a tangent line the radius $s$
and the centre-distance $1$ subtend the half-angle $\arcsin s$. Using the elementary
inequality
\begin{equation}\label{eq:arcsinineq}
   \arcsin s\le\frac{s}{1-s},\quad 0\le s<1,
\end{equation}
with $s=r/\mu<1$, we sharpen the last estimate to
\begin{equation}\label{eq:thetabound}
|\theta|\le \frac{s}{1-s}=\frac{r}{\mu-r}.
\end{equation}
Writing $x=\Re x+i\,\Im x$ and taking the real part of
$x\operatorname{Log}(1- t/\mu)=x(\log\rho+i\theta)$,
\begin{equation}\label{eq:powerfactor}
\left|\left(1-\frac{t}{\mu}\right)^{x}\right|
=\bigl|e^{x(\log\rho+i\theta)}\bigr|
=e^{\,\Re x\,\log\rho-\Im x\,\theta}
=\rho^{\,\Re x}\,e^{-\theta\,\Im x}.
\end{equation}
We bound the two factors separately. Since $\log\rho$ may have either sign, we
exploit the symmetry of the envelope~\eqref{eq:rhotheta}. From
\begin{equation*}
    (1+s)(1-s)=1-s^{2}\le1,
\end{equation*}
it follows that $1+s\le(1-s)^{-1}$, hence
\begin{equation}\label{eq:rhoenvelope}
    \frac{\mu-r}{\mu}=1-s\le\rho\le 1+s\le\frac{1}{1-s}=\frac{\mu}{\mu-r},
\end{equation}
and therefore
\begin{equation*}
    |\log\rho|\le\log\left(\frac{\mu}{\mu-r}\right).
\end{equation*}
Consequently,
\begin{equation}\label{eq:modpart}
\rho^{\,\Re x}=e^{(\Re x)\log\rho}\le e^{|\Re x|\,|\log\rho|}
\le\left(\frac{\mu}{\mu-r}\right)^{|\Re x|}.
\end{equation}
For the second factor, \eqref{eq:thetabound} gives
\begin{equation*}
    |\theta\,\Im x|=|\theta|\,|\Im x|\le \frac{r}{\mu-r}\,|\Im x|,
\end{equation*}
whence
\begin{equation}\label{eq:argpart}
e^{-\theta\,\Im x}\le e^{|\theta\,\Im x|}\le e^{\frac{r}{\mu-r}|\Im x|}.
\end{equation}
Combining $|e^{t}|\le e^{r}$ with~\eqref{eq:powerfactor}, \eqref{eq:modpart}
and~\eqref{eq:argpart} yields
\begin{equation*}
    \max_{|t|=r}\bigl|F(t,x)\bigr|
    \le e^{r}\left(\frac{\mu}{\mu-r}\right)^{|\Re x|}e^{\frac{r}{\mu-r}|\Im x|}=B_r(x),
\end{equation*}
and substitution into~\eqref{eq:cauchystep} gives~\eqref{eq:cauchybound}. Finally,
$B_r(x)$ is a power of $\tfrac{\mu}{\mu-r}$ in $|\Re x|$ times an exponential in
$|\Im x|$; it is therefore continuous in $x$ and independent of $n$.
\end{proof}

\begin{lemma}
\label{lem:pointwise}
Let $\{P_n\}$ be the monic Charlier family, $\alpha<0$, $j\ge1$. There exist
constants $0<c_0\le c$, depending only on $\alpha$, such that, for all $n\ge1$,
\begin{equation}\label{eq:Pmtwosided}
c_0\,\Gamma(n-\alpha)\le|P_n(\alpha)|\le c\,\Gamma(n-\alpha).
\end{equation}
\end{lemma}

\begin{proof}
	Taking absolute values in the Mehler--Heine formula~\eqref{eq:mhCclassical} and
	using $|\kappa_n|=1$, we obtain
	\begin{equation}\label{eq:Pmlimit}
		\frac{|P_n(\alpha)|}{\Gamma(n-\alpha)}\xrightarrow[n\to\infty]{}|\varphi_C(\alpha)|.
	\end{equation}
	Since $\alpha<0$, the argument $-\alpha$ is positive and is not a pole of $\Gamma$,
	so $\Gamma(-\alpha)\in(0,\infty)$ and hence $\varphi_C(\alpha)\neq0$. The convergent
	sequence~\eqref{eq:Pmlimit} therefore has the strictly positive limit
	$|\varphi_C(\alpha)|$, so there exist constants $0<a\le b$ and an index $n_0$,
	depending only on $\alpha$, with
	\begin{equation}\label{eq:tail}
		a\,\Gamma(n-\alpha)\le|P_n(\alpha)|\le b\,\Gamma(n-\alpha),\quad n\ge n_0 .
	\end{equation}
	It remains to extend~\eqref{eq:tail} to the finitely many indices $1\le n<n_0$.
	Since the Charlier orthogonality measure is positive and supported on
	$\{0,1,2,\ldots\}$, all zeros of the monic Charlier polynomial $P_n$ are real,
	simple, and belong to the convex hull of the support, namely $[0,\infty)$, see \cite{Ismail05}.
	Since $\alpha<0$, it follows that $P_n(\alpha)\neq0$ for every $n\ge1$.
	Consequently, the finite set
	\begin{equation*}
	    \left\{\frac{|P_n(\alpha)|}{\Gamma(n-\alpha)}:1\le n<n_0\right\},
	\end{equation*}
	consists of strictly positive numbers. Setting
	\begin{equation*}
	    c_0=\min\Bigl\{a,\,\min_{1\le n<n_0}\frac{|P_n(\alpha)|}{\Gamma(n-\alpha)}\Bigr\},
	\quad
	c=\max\Bigl\{b,\,\max_{1\le n<n_0}\frac{|P_n(\alpha)|}{\Gamma(n-\alpha)}\Bigr\},
	\end{equation*}
	we have $0<c_0\le c$, and~\eqref{eq:Pmtwosided} holds for all $n\ge1$.
\end{proof}

\begin{lemma}\label{lem:estimations}
Let $\alpha<0$ and $j\ge1$. Then there exists a constant
$c_3=c_3(\alpha,j)>0$ such that, for all $n>j+1$ and all $0\le k\le j$,
\begin{equation}\label{eq:gammaratiobound}
\frac{\dfrac{n!}{(n-j)!}\,\Gamma(n-j-\alpha)\cdot
\dfrac{(n-1)!}{(n-1-k)!}\,\Gamma(n-1-k-\alpha)}
{\Bigl(\dfrac{(n-1)!}{(n-1-j)!}\Bigr)^{2}\Gamma(n-1-j-\alpha)^{2}}
\le c_3\,(n+1)^{\,2j-\alpha}.
\end{equation}
\end{lemma}

\begin{proof}
Denote by $Q_{n,k}$ the left-hand side of~\eqref{eq:gammaratiobound}. We
estimate $Q_{n,k}$ by factoring it into a purely factorial part and a purely
Gamma part, each governed by the elementary asymptotics
\begin{equation}\label{eq:elemasymp}
\frac{n!}{(n-p)!}=n^{p}\bigl(1+O(n^{-1})\bigr),
\quad
\frac{\Gamma(n-a)}{\Gamma(n-b)}=n^{\,b-a}\bigl(1+O(n^{-1})\bigr),
\end{equation}
valid as $n\to\infty$ for fixed $p\in\N_0$ and fixed real $a,b$. Both relations
in~\eqref{eq:elemasymp} are the quantitative form of the Stirling ratio
asymptotics \cite{KLS2010},
\begin{equation}\label{eq:stirlingratio}
    \frac{\Gamma(z+a)}{\Gamma(z+b)}
\sim z^{a-b},
\quad a,b\in\mathbb{C},\quad |z|\to\infty.
\end{equation}
The second is its case $z=n$,
$(a,b)\mapsto(-a,-b)$, and the first follows upon writing
$n!/(n-p)!=\Gamma(n+1)/\Gamma(n-p+1)$ and applying~\eqref{eq:stirlingratio} with
$z=n$, $a=1$, $b=1-p$; in each case the next term of the expansion supplies the
relative error $O(n^{-1})$, with implied constant depending only on $p$,
respectively $a,b$. Since $0\le k\le j$, these parameters range over a finite
set, so the asymptotics hold uniformly in $k$.

The factorial part collects the two falling factorials of the numerator against
the squared one of the denominator. Applying the first relation
in~\eqref{eq:elemasymp} to each block,
\begin{equation*}
    \frac{\dfrac{n!}{(n-j)!}\cdot\dfrac{(n-1)!}{(n-1-k)!}}
{\Bigl(\dfrac{(n-1)!}{(n-1-j)!}\Bigr)^{2}}
=n^{\,j}\,(n-1)^{\,k-2j}\bigl(1+O(n^{-1})\bigr)
=O\!\bigl(n^{\,k-j}\bigr),
\end{equation*}
the exponent being $j+(k-2j)=k-j\le0$ because $k\le j$.

The Gamma part collects the two Gamma factors of the numerator against the
squared one of the denominator. Writing each ratio relative to the common
denominator $\Gamma(n-1-j-\alpha)$ and invoking the second relation
in~\eqref{eq:elemasymp},
\begin{equation*}
    \frac{\Gamma(n-j-\alpha)}{\Gamma(n-1-j-\alpha)}=n\bigl(1+O(n^{-1})\bigr),
\quad
\frac{\Gamma(n-1-k-\alpha)}{\Gamma(n-1-j-\alpha)}=n^{\,j-k}\bigl(1+O(n^{-1})\bigr),
\end{equation*}
where the second exponent records the $j-k\ge0$ unit increments carrying the
argument $n-1-j-\alpha$ up to $n-1-k-\alpha$. Multiplying,
\begin{equation*}
    \frac{\Gamma(n-j-\alpha)\,\Gamma(n-1-k-\alpha)}{\Gamma(n-1-j-\alpha)^{2}}
=O\!\bigl(n^{\,1+(j-k)}\bigr).
\end{equation*}
Combining the two parts, the variable exponents cancel exactly,
\begin{equation*}
    (k-j)+\bigl(1+(j-k)\bigr)=1,
\end{equation*}
so that $Q_{n,k}=O(n)$, uniformly in $0\le k\le j$. Finally, the hypotheses
$j\ge1$ and $\alpha<0$ yield $2j-\alpha\ge2-\alpha>1$, whence
$n\le(n+1)\le(n+1)^{\,2j-\alpha}$; thus $Q_{n,k}=O\bigl((n+1)^{\,2j-\alpha}\bigr)$.
Absorbing the implied constant into $c_3=c_3(\alpha,j)$ gives~\eqref{eq:gammaratiobound}.
\end{proof}

\begin{lemma}\label{lem:ABbound}
Let $\{P_n\}$ be monic Charlier family, $\alpha<0$, $j\ge1$, $\lambda>0$. For
fixed $x$ with $x-\alpha\notin\{0,1,\dots,j\}$ there is
$c_2=c_2(x,\alpha,j,\lambda)>0$ with
\begin{equation}\label{eq:ABbound}
\max\bigl\{|\Omega_n\Aop^{(j)}_n(x,\alpha)|,\,
|\Omega_n\Bop^{(j)}_n(x,\alpha)|\bigr\}
\le c_2\,(n+1)^{\,2j-\alpha},\quad n\ge1.
\end{equation}
\end{lemma}

\begin{proof}
Clearly, from \eqref{eq:Kij}, we have
\begin{equation*}
    \Kn^{(j,j)}_{n-1}(\alpha,\alpha)
=\sum_{k=0}^{n-1}\frac{\bigl(\dif^{j}P_k(\alpha)\bigr)^2}{\norm{P_k}^2}
\ge\frac{\bigl(\dif^{j}P_{n-1}(\alpha)\bigr)^2}{\norm{P_{n-1}}^2}.
\end{equation*}
Since $\delta_n\ge\lambda\,\Kn^{(j,j)}_{n-1}(\alpha,\alpha)$, this yields
\begin{equation}\label{eq:deltalower}
\delta_n\ge\frac{\lambda\,\bigl(\dif^{j}P_{n-1}(\alpha)\bigr)^2}
{\norm{P_{n-1}}^2}.
\end{equation}
Next we bound $\Aop^{(j)}_n$ itself. From the closed form~\eqref{eq:Aj}, the
triangle inequality together with the elementary estimate
\begin{equation*}
    \bigl|\sum_k a_k d_k\bigr|\le\bigl(\sum_k|a_k|\bigr)\max_k|d_k|,
\end{equation*}
applied with
$d_k=\dif^kP_{n-1}(\alpha)$, gives
\begin{equation}\label{eq:Abound}
|\Aop^{(j)}_n(x,\alpha)|
\le\frac{C_x}{\norm{P_{n-1}}^{2}}\,
\max_{0\le k\le j}\bigl|\dif^{k}P_{n-1}(\alpha)\bigr|,
\quad
C_x=j!\sum_{k=0}^{j}\frac{|\ff{x-\alpha}{k}|}{k!\,|\ff{x-\alpha}{j+1}|},
\end{equation}
the constant $C_x$ being finite precisely because $\ff{x-\alpha}{j+1}\neq0$.

These two estimates combine to produce the announced cancellation. Multiplying
$|\Omega_n|$ by the bound~\eqref{eq:Abound} and
using~\eqref{eq:deltalower} in the denominator, the two occurrences of
$\norm{P_{n-1}}^2$ cancel, leaving
\begin{equation}\label{eq:productcancel}
|\Omega_n\Aop^{(j)}_n(x,\alpha)|
\le\frac{C_x}{\lambda}\,
\frac{|\dif^{j}P_n(\alpha)|}{\bigl(\dif^{j}P_{n-1}(\alpha)\bigr)^{2}}\,\max_{0\le k\le j}
|\dif^{k}P_{n-1}(\alpha)|.
\end{equation}
This is the structural core of the argument: the factorially small weight
$1/\norm{P_{n-1}}^2$ carried by $\Aop^{(j)}_n$ is absorbed by the factorially
large lower bound for $\delta_n$, so that only ratios of difference values at the
mass point survive.

Fix such an $n$ and let $k^{*}=k^{*}(n)\in\{0,\dots,j\}$ attain the maximum
$\max_{0\le k\le j}|\dif^{k}P_{n-1}(\alpha)|$. Inserting the forward shift
operator~\eqref{eq:fwdgeneral} into each of the three difference values
in~\eqref{eq:productcancel} and bounding each polynomial value by the
appropriate side of~\eqref{eq:Pmtwosided}, we obtain
\begin{align*}
|\dif^{j}P_n(\alpha)|
&=\frac{n!}{(n-j)!}\,|P_{n-j}(\alpha)|
\le c\,\frac{n!}{(n-j)!}\,\Gamma(n-j-\alpha),\\[2pt]
|\dif^{k^{*}}P_{n-1}(\alpha)|
&=\frac{(n-1)!}{(n-1-k^{*})!}\,|P_{n-1-k^{*}}(\alpha)|
\le c\,\frac{(n-1)!}{(n-1-k^{*})!}\,\Gamma(n-1-k^{*}-\alpha),\\[2pt]
\bigl(\dif^{j}P_{n-1}(\alpha)\bigr)^{2}
&=\Bigl(\frac{(n-1)!}{(n-1-j)!}\Bigr)^{2}P_{n-1-j}(\alpha)^{2}
\ge c_0'^{\,2}\Bigl(\frac{(n-1)!}{(n-1-j)!}\Bigr)^{2}\Gamma(n-1-j-\alpha)^{2},
\end{align*}
where the first two lines use the upper bound of~\eqref{eq:Pmtwosided} (with
$m=n-j$ and $m=n-1-k^{*}$, respectively) and the third uses its lower bound
(with $m=n-1-j$). Since the numerator is bounded above by a positive quantity and
the denominator below by a positive quantity, division is legitimate, and the
ratio in~\eqref{eq:productcancel} satisfies
\begin{equation*}
    \frac{|\dif^{j}P_n(\alpha)|}
{\bigl(\dif^{j}P_{n-1}(\alpha)\bigr)^{2}}\,\max_{0\le k\le j}|\dif^{k}P_{n-1}(\alpha)|
\le\frac{c^{2}}{c_0'^{\,2}}\;
\frac{\dfrac{n!}{(n-j)!}\,\Gamma(n-j-\alpha)\cdot
\dfrac{(n-1)!}{(n-1-k^{*})!}\,\Gamma(n-1-k^{*}-\alpha)}
{\Bigl(\dfrac{(n-1)!}{(n-1-j)!}\Bigr)^{2}\Gamma(n-1-j-\alpha)^{2}}.
\end{equation*}
The second factor is exactly the left-hand side of~\eqref{eq:gammaratiobound}
with $k=k^{*}$, see Lemma \ref{lem:estimations}. As that bound holds for every $0\le k\le j$, it holds in
particular at $k=k^{*}$ and bounds this factor by $c_3\,(n+1)^{\,2j-\alpha}$.
Substituting into~\eqref{eq:productcancel},
\begin{equation*}
    |\Omega_n\Aop^{(j)}_n(x,\alpha)|
\le\frac{C_x}{\lambda}\cdot\frac{c^{2}}{c_0'^{\,2}}\cdot
c_3\,(n+1)^{\,2j-\alpha}
=c_2'\,(n+1)^{\,2j-\alpha},\quad n\ge j+1.
\end{equation*}
For the finitely many indices $1\le n\le j$, the quantity
$|\Omega_n\Aop^{(j)}_n(x,\alpha)|$ is finite, because the denominator
$\delta_n\ge1$ is bounded below by its
constant term, while $|\Aop^{(j)}_n(x,\alpha)|<\infty$ since $C_x<\infty$ under
the standing hypothesis $\ff{x-\alpha}{j+1}\neq0$. Hence
\begin{equation*}
    M=\max_{1\le n\le j}
\frac{|\Omega_n\Aop^{(j)}_n(x,\alpha)|}{(n+1)^{\,2j-\alpha}}<\infty,
\end{equation*}
and $c_2=\max\{c_2',M\}$ extends the bound to all $n\ge1$.

Finally we turn to the $\Bop$-multiplier. The closed form~\eqref{eq:Bj} differs
from~\eqref{eq:Aj} only by an overall sign and by carrying
$\dif^{k}P_{n}(\alpha)$ rather than $\dif^{k}P_{n-1}(\alpha)$. The sign is
immaterial to absolute values, and repeating the argument verbatim replaces the
factor $\max_{0\le k\le j}|\dif^{k}P_{n-1}(\alpha)|$ in~\eqref{eq:productcancel}
by $\max_{0\le k\le j}|\dif^{k}P_{n}(\alpha)|$. At the corresponding maximiser,
the forward shift operator with $m=n$ and the upper bound
of~\eqref{eq:Pmtwosided} bound this term by
$c\,\tfrac{n!}{(n-k)!}\,\Gamma(n-k-\alpha)$. In the bookkeeping ratio this
amounts to the index shift $n-1\mapsto n$ in the second numerator block: the
falling factorial $\tfrac{(n-1)!}{(n-1-k)!}=O(n^{k})$ is replaced by
$\tfrac{n!}{(n-k)!}=O(n^{k})$, and $\Gamma(n-1-k-\alpha)$ by
$\Gamma(n-k-\alpha)=O(n)\,\Gamma(n-1-k-\alpha)$. The same powers of $n$ govern
both, so the worst-case exponent is unchanged and only the implied constant
grows. The bound~\eqref{eq:gammaratiobound}, applied with this shifted block,
again yields an $O\bigl((n+1)^{2j-\alpha}\bigr)$ estimate for $n\ge j+1$, and the
indices $1\le n\le j$ are absorbed as before. Thus
$|\Omega_n\Bop^{(j)}_n(x,\alpha)|\le c_2''\,(n+1)^{2j-\alpha}$ for all $n\ge1$,
and taking $c_2=\max\{c_2',c_2'',M\}$ gives~\eqref{eq:ABbound}.
\end{proof}

The proofs of the preceding three lemmas extend directly to the Meixner
family after replacing the corresponding Charlier quantities by their
Meixner counterparts. More precisely, one uses the radius of convergence
$\mu$ of the generating function~\eqref{eq:GMmonic}, the squared norm
\begin{equation*}
    \|P_n\|^2=n!(\gamma)_n\mu^n(1-\mu)^{-\gamma-2n},
\end{equation*}
and the Meixner Mehler--Heine formula~\eqref{eq:mhMclassical} in place of the
corresponding Charlier formulas. No further modifications of the proofs
are required. Although the family-dependent constants and exponential
growth factors change, the polynomial dependence on $n$ remains the same
in both cases. For this reason, we present the following convergence
result simultaneously for the Charlier and Meixner families.

\begin{proposition}\label{prop:convergence}
	Fix $x\in\C$, $j\ge1$, $\lambda>0$, $\alpha<0$. For either family the
	monic-normalised Sobolev correction
	\begin{equation}\label{eq:Rexp}
		R_{j,\lambda,\alpha}(x;t)
		=\sum_{n\ge1}\frac{\psi_n\,\Omega_n\,\Kn^{(0,j)}_{n-1}(x,\alpha)}{n!}\,t^{n},
	\end{equation}
	converges absolutely for $|t|<\mu$ and therefore defines a function analytic on
	that disc, where $\mu$ is the radius of~\eqref{eq:GCmonic} (Charlier)
	or~\eqref{eq:GMmonic} (Meixner).
\end{proposition}

\begin{proof}
	We establish absolute convergence of~\eqref{eq:Rexp} by a root test on its
	coefficients
    \begin{equation*}
        c_n=\frac{\psi_n\Omega_n\Kn^{(0,j)}_{n-1}(x,\alpha)}{n!},
    \end{equation*}
    analyticity on
	$|t|<\mu$ then follows, since an absolutely convergent power series is analytic
	inside its disc of convergence. The decisive step is to pair the Sobolev
	multiplier $\Omega_n$ with each kernel coefficient before estimating, so that the
	two preceding lemmas act on their natural arguments. By the kernel
	representation~\eqref{eq:K0j},
	\begin{equation*}
	    \Omega_n\Kn^{(0,j)}_{n-1}(x,\alpha)
	=\bigl(\Omega_n\Aop^{(j)}_n\bigr)P_n(x)
	+\bigl(\Omega_n\Bop^{(j)}_n\bigr)P_{n-1}(x).
	\end{equation*}
	Fix $0<r<\mu$. The Cauchy bound of Lemma~\ref{lem:cauchy} controls the
	polynomial values against the normalising sequence: in the Charlier case the
	geometric factor $|\psi_n|=\mu^{-n}$ cancels exactly the $\mu^{n}$ carried by
	$|P_n(x)|$, so that, up to a constant absorbed into $B_r(x)$,
	\begin{equation*}
	    |\psi_n|\,|P_n(x)|\le B_r(x)\,n!\,r^{-n}.
	\end{equation*}
	The factor $P_{n-1}(x)$ requires converting the bound at index $n-1$ into the
	format of index $n$. Clearly,
    \begin{equation*}
        |\psi_n|\,|P_{n-1}(x)|
	=\frac{|\psi_n|}{|\psi_{n-1}|}\,|\psi_{n-1}|\,|P_{n-1}(x)|.
    \end{equation*}
    Applying Lemma~\ref{lem:cauchy} at index $n-1$, we deduce
	\begin{equation*}
	    |\psi_n|\,|P_{n-1}(x)|
	\le \frac{|\psi_n|}{|\psi_{n-1}|}\,B_r(x)\,(n-1)!\,r^{-(n-1)}
	=\frac{r}{\mu\,n}\,B_r(x)\,n!\,r^{-n}\le B_r(x)\,n!\,r^{-n},
	\end{equation*}
	where the second equality uses the geometric ratio
	$|\psi_n|/|\psi_{n-1}|=\mu^{-1}$ (Charlier) together with
	$(n-1)!=n!/n$ and $r^{-(n-1)}=r\cdot r^{-n}$, and the final inequality holds
	because $r/(\mu n)\le r/\mu<1$ for $n\ge1$ and $0<r<\mu$. Thus both polynomial
	factors obey the common bound
    \begin{equation*}
        |\psi_n|\,|P(x)|\le B_r(x)\,n!\,r^{-n},
    \end{equation*}
    with the
	adjustment for $P_{n-1}$ only improving the constant. Combining this with the
	product bound
	$\max\{|\Omega_n\Aop^{(j)}_n|,|\Omega_n\Bop^{(j)}_n|\}\le c_2\,(n+1)^{2j-\alpha}$
	of Lemma~\ref{lem:ABbound}, the triangle inequality gives
	\begin{equation}\label{eq:cnfinal}
		|c_n|\le C(x,\alpha,j,\lambda,r)\,(n+1)^{2j-\alpha}\,r^{-n},
		\quad C=2c_2B_r(x).
	\end{equation}
	The polynomial prefactor $(n+1)^{2j-\alpha}$ has $n$-th root tending to $1$, so
	it leaves the root test unaffected, and~\eqref{eq:cnfinal} yields
    \begin{equation*}
        \limsup_n\sqrt[n]{|c_n|}\le\frac{1}{r}.
    \end{equation*}
    Since $0<r<\mu$ was arbitrary, taking the
	supremum over all such $r$ gives
    \begin{equation*}
        \limsup_n\sqrt[n]{|c_n|}\le\frac{1}{\mu},
    \end{equation*}
    equivalently,
	the radius of convergence of~\eqref{eq:Rexp} is at least $\mu$, so the series
	converges absolutely on $|t|<\mu$. The Meixner case is identical, the sole
	change being the geometric factor
    \begin{equation*}
        |\psi_n|= \left(\frac{1-\mu}{\mu}\right)^{n}.
    \end{equation*}
    Hence, the ratio
    \begin{equation*}
        \frac{|\psi_n|}{|\psi_{n-1}|}=\frac{(1-\mu)}{\mu},
    \end{equation*}
    is
	again a constant, so the same adjustment produces~\eqref{eq:cnfinal}.
\end{proof}

\begin{theorem}[Formal Sobolev generating function]\label{thm:formalGF}
	Let $\{P_n\}_{n\ge0}$ be either the monic Charlier or the monic Meixner
	family and $\{\Sob{n}\}_{n\ge0}$ the Sobolev-type family
	of~\eqref{eq:innerprod}. Then, for $|t|<\mu$,
	\begin{equation}\label{eq:GSfirst}
		G_S(x;t)=G_P(x;t)-\lambda\,R_{j,\lambda,\alpha}(x;t),
	\end{equation}
	where, in the monic-normalised convention~\eqref{eq:GPmonic}, the Sobolev
	correction series carries the same weight $\frac{\psi_n}{n!}$,
	\begin{equation}\label{eq:Rdef}
		R_{j,\lambda,\alpha}(x;t)
		=\sum_{n\ge1}\frac{\psi_n}{n!}\,\Omega_n\,\Kn_{n-1}^{(0,j)}(x,\alpha)\,t^{n}.
	\end{equation}
	In particular $G_S(x;t)$ is analytic on $|t|<\mu$, the same disc on which
	$G_P(x;t)$ is analytic. Equivalently, using the kernel
	representation~\eqref{eq:K0j},
	\begin{equation}\label{eq:GStwomult}
		G_S(x;t)=G_P(x;t)
		-\lambda\sum_{n\ge1}\frac{\psi_n}{n!}\,\Omega_n \Aop_n^{(j)}(x,\alpha)P_n(x)\,t^{n}
		-\lambda t\sum_{n\ge0}\frac{\psi_{n+1}}{(n+1)!}\,\Omega_{n+1}\Bop_{n+1}^{(j)}(x,\alpha)P_n(x)\,t^{n}.
	\end{equation}
\end{theorem}

\begin{proof}
	Multiply the connection formula~\eqref{eq:Omega} by $\psi_n t^{n}/n!$, and sum over
	$n\ge1$. Adding the $n=0$ term, which contributes $\psi_0\Sob{0}=\psi_0P_0$ with
	no correction since $\Sob{0}=P_0$, the monic-normalised Sobolev generating
	function splits as
	\begin{equation*}
	    G_S(x;t)=\sum_{n\ge0}\frac{\psi_n}{n!}\,\Sob{n}(x)\,t^{n}
	=G_P(x;t)-\lambda\sum_{n\ge1}\frac{\psi_n}{n!}\,\Omega_n\Kn_{n-1}^{(0,j)}(x,\alpha)\,t^{n},
	\end{equation*}
	which is precisely~\eqref{eq:GSfirst} with $R_{j,\lambda,\alpha}$ as
	in~\eqref{eq:Rdef}. By Proposition~\ref{prop:convergence} the correction series
	converges absolutely on $|t|<\mu$, where $G_P$ is analytic by~\eqref{eq:GCmonic};
	hence the rearrangement is legitimate and, being a difference of two functions
	analytic on that disc, $G_S(x;t)$ is itself analytic on $|t|<\mu$. All
	manipulations below take place within this common disc.
	
	To obtain~\eqref{eq:GStwomult}, substitute the kernel
	representation~\eqref{eq:K0j} with $y=\alpha$,
	\begin{equation*}
	    \Omega_n\Kn_{n-1}^{(0,j)}(x,\alpha)
	=\Omega_n\Aop_n^{(j)}(x,\alpha)\,P_n(x)
	+\Omega_n\Bop_n^{(j)}(x,\alpha)\,P_{n-1}(x),
	\end{equation*}
	into~\eqref{eq:Rdef}, splitting $R_{j,\lambda,\alpha}$ into an $\Aop$-sum and a
	$\Bop$-sum. The $\Aop$-sum already has the form displayed
	in~\eqref{eq:GStwomult}. In the $\Bop$-sum, the reindexing $n=m+1$ (so that
	$P_{n-1}=P_m$ and $m$ runs over $n\ge0$) extracts one power of $t$ and transforms
	the summand into
	\begin{equation*}
	 \frac{\psi_{m+1}}{(m+1)!}\,\Omega_{m+1}\Bop_{m+1}^{(j)}(x,\alpha)\,P_m(x)\,t^{m},
	\end{equation*}
	with the shifted weight $\frac{\psi_{m+1}}{(m+1)!}$. Absolute convergence again justifies
	treating the two sums separately, and renaming $m$ as $n$
	yields~\eqref{eq:GStwomult}.
\end{proof}

\begin{example}[Degree-two Sobolev polynomials via the generating function]
\label{ex:degree2}
We illustrate Theorem~\ref{thm:formalGF} by extracting $\Sob{2}$ from the
coefficient of $t^{2}$ in $G_S(x;t)$, for both families, with $j=1$. Throughout,
$G_P(x;t)$ denotes the monic-normalised generating function~\eqref{eq:GPmonic},
$[t^{n}]f$ denotes the coefficient of $t^{n}$ in the Taylor expansion of $f$ about
$t=0$, and the correction series~\eqref{eq:Rdef} carries the same weight
$\psi_n/n!$. Since $[t^{n}]G_P(x;t)=(\psi_n/n!)\,P_n(x)$ and
\begin{equation*}
    [t^{n}]R_{1,\lambda,\alpha}(x;t)=(\psi_n/n!)\,\Omega_n\,\Kn^{(0,1)}_{n-1}(x,\alpha),
\end{equation*}
the
common scalar $\psi_n/n!$ factors out of~\eqref{eq:GSfirst}, giving
\begin{equation}\label{eq:extract}
\Sob{n}(x)=\frac{n!}{\psi_n}\,[t^{n}]\,G_S(x;t)
=P_n(x)-\lambda\,\Omega_n\,\Kn^{(0,1)}_{n-1}(x,\alpha),
\end{equation}
which is the connection formula~\eqref{eq:ConxF1}. The extraction is therefore
independent of the particular weight $\psi_n$: the normalisation chosen for $G_P$ is
inherited verbatim by the correction series~\eqref{eq:Rdef}, so it cancels in the
extracted coefficient.

We illustrate Theorem~\ref{thm:formalGF} by extracting $\Sob{2}$ from the
coefficient of $t^{2}$ in $G_S(x;t)$, for both families, with $j=1$. Throughout,
\eqref{eq:GPmonic} denotes the monic-normalised generating
function~\eqref{eq:GPmonic}, and the correction series~\eqref{eq:Rdef} carries the
same weight $\frac{\psi_n}{n!}$. Since $[t^{n}]G_P=(\frac{\psi_n}{n!})P_n$ and
$[t^{n}]R_{1,\lambda,\alpha}=(\frac{\psi_n}{n!})\Omega_n\Kn^{(0,1)}_{n-1}(x,\alpha)$, the
common scalar $\frac{\psi_n}{n!}$ factors out of~\eqref{eq:GSfirst}, giving
\begin{equation}\label{eq:extract}
\Sob{n}(x)=\frac{n!}{\psi_n}\,[t^{n}]\,G_S(x;t)
=P_n(x)-\lambda\,\Omega_n\,\Kn^{(0,1)}_{n-1}(x,\alpha),
\end{equation}
which is the connection formula~\eqref{eq:ConxF1}. The extraction is therefore
independent of the particular weight $\psi_n$.
\end{example}

\subsection*{Example 1: Charlier Sobolev type}
The monic Charlier polynomials and squared norms are
\begin{equation*}
C^{(\mu)}_0=1,\quad
C^{(\mu)}_1=x-\mu,\quad
C^{(\mu)}_2=x^{2}-(2\mu+1)x+\mu^{2}.
\end{equation*}
The monic-normalised generating function, to order $t^{2}$, reads
\begin{equation*}\label{eq:GPcharlier}
G_P(x;t)=\psi_0+\psi_1(x-\mu)\,t
+\frac{\psi_2}{2}\bigl(x^{2}-(2\mu+1)x+\mu^{2}\bigr)\,t^{2}+O(t^{3}).
\end{equation*}
From the forward-shift relation $\dif C^{(\mu)}_n=n\,C^{(\mu)}_{n-1}$, the
differences at $\alpha$ are
\begin{equation*}
\dif C^{(\mu)}_1(\alpha)=1,
\quad
\dif C^{(\mu)}_2(\alpha)=2\,C^{(\mu)}_1(\alpha)=2(\alpha-\mu),
\end{equation*}
whence the surviving $k=1$ term yields
\begin{equation*}\label{eq:charlierdata}
\Kn^{(1,1)}_{1}(\alpha,\alpha)=\frac{1}{\mu},
\quad
\Omega_2=\frac{2\mu(\alpha-\mu)}{\mu+\lambda},
\quad
\Kn^{(0,1)}_{1}(x,\alpha)=\frac{x-\mu}{\mu}.
\end{equation*}
The correction series~\eqref{eq:Rdef}, to order $t^{2}$, is
\begin{equation*}
R_{1,\lambda,\alpha}(x;t)
=\frac{\psi_1}{1!}\,\Omega_1\,\Kn^{(0,1)}_{0}(x,\alpha)\,t
+\frac{\psi_2}{2!}\,\Omega_2\,\Kn^{(0,1)}_{1}(x,\alpha)\,t^{2}+O(t^{3}),
\end{equation*}
whose coefficient of $t^{2}$ is $\tfrac{\psi_2}{2}\,\Omega_2\,(x-\mu)/\mu$.
Subtracting $\lambda R_{1,\lambda,\alpha}$ from $G_P$ as in~\eqref{eq:GSfirst} and
reading off the coefficient of $t^{2}$,
\begin{equation*}
[t^{2}]\,G_S(x;t)
=\frac{\psi_2}{2}\Bigl[\,C^{(\mu)}_2(x)-\lambda\,\Omega_2\,\frac{x-\mu}{\mu}\Bigr],
\end{equation*}
so that, by~\eqref{eq:extract},
\begin{equation*}\label{eq:S2charliergen}
\Sob{2}(x)
=C^{(\mu)}_2(x)-\lambda\,\Omega_2\,\frac{x-\mu}{\mu}
=x^{2}-\Bigl(2\mu+1+\frac{2\lambda(\alpha-\mu)}{\mu+\lambda}\Bigr)x
+\mu^{2}+\frac{2\lambda\mu(\alpha-\mu)}{\mu+\lambda}.
\end{equation*}

\subsection*{Example 2: Meixner Sobolev type}
The monic Meixner polynomials and squared norms are
\begin{equation*}
M^{(\gamma,\mu)}_0(x)=1,
\quad
M^{(\gamma,\mu)}_1(x)=x-\frac{\gamma\mu}{1-\mu},
\end{equation*}
and, from the three-term recurrence,
\begin{equation*}
M^{(\gamma,\mu)}_2(x)
=x^{2}+\frac{\mu^{2}(2\gamma+1)-2\gamma\mu-1}{(1-\mu)^{2}}\,x
+\frac{\mu^{2}\gamma(\gamma+1)}{(1-\mu)^{2}} .
\end{equation*}
The monic-normalised generating function, to order $t^{2}$, reads
\begin{equation*}\label{eq:GPmeixner}
G_P(x;t)=\psi_0+\psi_1\Bigl(x-\frac{\gamma\mu}{1-\mu}\Bigr)t
+\frac{\psi_2}{2}\,M^{(\gamma,\mu)}_2(x)\,t^{2}+O(t^{3}).
\end{equation*}
From the forward-shift relation $\dif M^{(\gamma,\mu)}_n=n\,M^{(\gamma+1,\mu)}_{n-1}$,
which raises the first parameter, the differences at $\alpha$ are
\begin{equation*}
\dif M^{(\gamma,\mu)}_1(\alpha)=1,
\quad
\dif M^{(\gamma,\mu)}_2(\alpha)
=2\,M^{(\gamma+1,\mu)}_1(\alpha)
=2x-\frac{2\mu(\gamma+1)}{1-\mu},
\end{equation*}
whence, the surviving
$k=1$ term yields
\begin{equation*}
    K_1^{(1,1)}(\alpha,\alpha)=\frac{(1-\mu)^{\gamma+2}}{\gamma\mu},
\quad
K_1^{(0,1)}(x,\alpha)
=-\frac{(1-\mu)^{\gamma+1}}{\gamma\mu}\bigl(\mu(x+\gamma)-x\bigr).
\end{equation*}
Second, both kernels remain independent of $\alpha$, exactly as in the Charlier
case, and the mixed kernel is again a polynomial of degree one; the cancellation
of $[x-\alpha]_2$ is therefore a feature common to both families. The
coefficient $\Omega_2$ reads
\begin{equation*}
    \Omega_2=\frac{2\gamma\mu\bigl[\alpha(\mu-1)+\mu(\gamma+1)\bigr]}
{(\mu-1)\bigl[\gamma\mu+\lambda(1-\mu)^{\gamma+2}\bigr]}.
\end{equation*}
Proceeding exactly as before, the coefficient of $t^{2}$ in
$G_S=G_P-\lambda R_{1,\lambda,\alpha}$ gives
\begin{equation*}
    \Sob{2}(x)=x^2+A_S x+B_S,
\end{equation*}
with
\begin{equation*}
    A_S=\frac{\mu^2(2\gamma+1)-2\mu\gamma-1}{(\mu-1)^2}
-\frac{2\lambda(1-\mu)^{\gamma+2}\bigl[\alpha(\mu-1)+\mu(\gamma+1)\bigr]}
{(\mu-1)\bigl[\gamma\mu+\lambda(1-\mu)^{\gamma+2}\bigr]},
\end{equation*}
\begin{equation*}
    B_S=\frac{\mu^2\gamma(\gamma+1)}{(\mu-1)^2}
-\frac{2\lambda\mu\gamma(1-\mu)^{\gamma+2}\bigl[\alpha(\mu-1)+\mu(\gamma+1)\bigr]}
{(\mu-1)^2\bigl[\gamma\mu+\lambda(1-\mu)^{\gamma+2}\bigr]}.
\end{equation*}
In both families the Sobolev perturbation preserves the monic normalisation and
alters only the lower-order coefficients, through a correction proportional to
$\lambda$ and vanishing linearly as $\lambda\to0$; the classical polynomial is
then recovered. The Charlier correction vanishes in addition at $\alpha=\mu$, and
the Meixner correction at $\alpha=\mu(\gamma+1)/(1-\mu)$, the respective zeros of
$\dif P_2(\alpha)$.

The cancellation of the weight $\psi_n/n!$ in~\eqref{eq:extract} explains why the
generating-function route and the connection formula~\eqref{eq:ConxF1} return the
same polynomial: the normalisation chosen for $G_P$ is inherited verbatim by the
correction series~\eqref{eq:Rdef}, so it plays no role in the extracted coefficient.

\section{Mehler--Heine type formulas}\label{sec:mehlerheine}

We establish Mehler--Heine type formulas for the Sobolev-type families
$\{\Sob{n}\}_{n\ge0}$. Since the mass point $\alpha$ lies outside the support
$[0,+\infty)$, the Sobolev perturbation is not asymptotically negligible under
the classical scaling: it cancels the leading term, shifts the Mehler--Heine scale
by one factor of $n$, and produces a new zero of the limit function at $z=\alpha$.
The resulting limit is independent of $\lambda$ and of $j$.

\begin{lemma}\label{lem:mhproperties}
Let $\varphi$ and $\kappa_n$ be as in~\eqref{eq:mhcommon}, and set $c_\mu=1$ for
Charlier, $c_\mu=1-\mu$ for Meixner. Then, for fixed $w\in\C\setminus[0,+\infty)$,
\begin{equation}\label{eq:mhratios}
\frac{P_n(w)}{P_{n-1}(w)}=-\,\frac{n-1-w}{c_\mu}\bigl(1+O(n^{-1})\bigr),
\quad
\frac{\kappa_n}{\kappa_{n-1}}=-\frac{1}{c_\mu}.
\end{equation}
\end{lemma}

\begin{proof}
In fact, the Mehler--Heine limit~\eqref{eq:mhcommon} at indices
$n$ and $n-1$ reads
\begin{equation*}
    P_n(w)=\kappa_n\,\Gamma(n-w)\,\varphi(w)\bigl(1+O(n^{-1})\bigr),
\quad
P_{n-1}(w)=\kappa_{n-1}\,\Gamma(n-1-w)\,\varphi(w)\bigl(1+O(n^{-1})\bigr).
\end{equation*}
Then, using the functional equation
$\Gamma(n-w)=(n-1-w)\Gamma(n-1-w)$, we deduce
\begin{equation*}
    \frac{P_n(w)}{P_{n-1}(w)}
=\frac{\kappa_n}{\kappa_{n-1}}\,(n-1-w)\bigl(1+O(n^{-1})\bigr).
\end{equation*}
From~\eqref{eq:mhcommon}, $\kappa_n/\kappa_{n-1}=-1$ for
Charlier and $\kappa_n/\kappa_{n-1}=(\mu-1)^{-1}$ for Meixner; since
$c_\mu=1$ and $c_\mu=1-\mu$ respectively, both equal $-c_\mu^{-1}$, which yields
the first identity in~\eqref{eq:mhratios}.
\end{proof}

\begin{lemma}\label{lem:bracket}
Let $\alpha\in\R_-$, $j\ge1$, and set
\begin{equation}\label{eq:Ndef}
    \Nf_n(z)=\dif^{j}P_{n-1}(\alpha)\,P_n(z)-\dif^{j}P_n(\alpha)\,P_{n-1}(z).
\end{equation}
Then, uniformly on compact subsets of $\C$,
\begin{equation}\label{eq:bracketover}
\frac{\Nf_n(z)}{\dif^{j}P_{n-1}(\alpha)}
=-\,\frac{z-\alpha}{n}\,\kappa_n\,\Gamma(n-z)\,\varphi(z)\,\bigl(1+o(1)\bigr),
\quad n\to\infty.
\end{equation}
\end{lemma}

\begin{proof}
By the forward-shift identity~\eqref{eq:fwdgeneral} and the consecutive-index
ratio~\eqref{eq:mhratios} applied to $\{\widetilde P_m\}$, we arrive
\begin{equation*}
    \frac{\dif^{j}P_n(\alpha)}{\dif^{j}P_{n-1}(\alpha)}
=\frac{[n]_j}{[n-1]_j}\,
\frac{\widetilde P_{n-j}(\alpha)}{\widetilde P_{n-1-j}(\alpha)}
=-\frac{1}{c_\mu}\,\frac{n}{n-j}\,(n-1-j-\alpha)\bigl(1+o(1)\bigr).
\end{equation*}
Writing $P_m(z)=\kappa_m\Gamma(m-z)\varphi(z)(1+o(1))$ and using
\eqref{eq:mhratios}, we deduce
\begin{equation*}
    \frac{\Nf_n(z)}{\dif^{j}P_{n-1}(\alpha)}
=P_n(z)-\frac{\dif^{j}P_n(\alpha)}{\dif^{j}P_{n-1}(\alpha)}\,P_{n-1}(z)
=\kappa_n\Gamma(n-z)\varphi(z)
\Bigl[1-\frac{n(n-1-j-\alpha)}{(n-j)(n-1-z)}\Bigr](1+o(1)).
\end{equation*}
After straightforward calculations, we obtain
\begin{equation*}
    \dfrac{n(n-1-j-\alpha)}{(n-j)(n-1-z)}=1+\dfrac{z-\alpha}{n}+O(n^{-2}),
\end{equation*}
the
bracket equals $-(z-\alpha)n^{-1}(1+o(1))$, which is~\eqref{eq:bracketover}.
\end{proof}

\begin{lemma}\label{lem:kernelgrowth}
Let $\alpha\in\R_-$ and $j\ge1$. Then $\Kn^{(j,j)}_{n-1}(\alpha,\alpha)\to\infty$ as
$n\to\infty$. More precisely, there exist $q>1$ and $n_0\in\N$, depending only on
the family and on $j$, such that
\begin{equation}\label{eq:kernelgrowth}
\Kn^{(j,j)}_{n-1}(\alpha,\alpha)\ge q^{n},\quad n\ge n_0.
\end{equation}
In particular, for each fixed $p>0$ one has
$n^{-p}\,\Kn^{(j,j)}_{n-1}(\alpha,\alpha)\to\infty$ as $n\to\infty$.
\end{lemma}

\begin{proof}
By the definition~\eqref{eq:Kij} of the iterated kernel and the forward shift
operator~\eqref{eq:fwdgeneral}, the differences $\dif^{j}P_k(\alpha)=[k]_j\widetilde
P_{k-j}(\alpha)$ vanish for $k<j$, so that
\begin{equation}\label{eq:Kjjdescent}
\Kn^{(j,j)}_{n-1}(\alpha,\alpha)=\sum_{k=j}^{n-1}t_k,
\quad
t_k=\frac{\bigl(\dif^{j}P_k(\alpha)\bigr)^2}{\norm{P_k}^2}=\frac{[k]_j^{2}\,\widetilde P_{k-j}(\alpha)^2}{\norm{P_k}^2}>0.
\end{equation}

We first establish that the series diverges. Since the summands are positive, it
suffices to show $t_k\to\infty$, which we do through the ratio $t_{k+1}/t_k$. The
Mehler--Heine limit~\eqref{eq:mhcommon} for the descended family gives
\begin{equation}\label{eq:Ptildeasympt}
\widetilde P_{k-j}(\alpha)=\kappa_{k-j}\,\Gamma(k-j-\alpha)\,\widetilde\varphi(\alpha)
\bigl(1+o(1)\bigr),
\end{equation}
where $\widetilde\varphi(\alpha)\neq0$ precisely because $\alpha\in\R_-$ avoids the
nonnegative integers, the only zeros of $\widetilde\varphi$; this is the sole place
where the hypothesis $\alpha<0$ is used. Squaring~\eqref{eq:Ptildeasympt} and
forming the ratio of consecutive terms, the three structural factors of $t_k$
evolve as follows. The falling factorial satisfies
\begin{equation*}
\frac{[k+1]_j}{[k]_j}=\frac{k+1}{k+1-j}\to1,
\end{equation*}
so its square contributes a factor tending to $1$. The Gamma factor evolves through
the functional equation
\begin{equation*}
\frac{\Gamma(k+1-j-\alpha)}{\Gamma(k-j-\alpha)}=k-j-\alpha,
\end{equation*}
contributing $(k-j-\alpha)^2$ after squaring. The remaining factors, the scale
$\kappa_{k-j}^2$ and the squared norm $\norm{P_k}^2$, are read from
\eqref{eq:mhcommon} and Table~\ref{tab:families}, respectively. Therefore,
\begin{equation*}
\frac{\norm{P_{k+1}}^2}{\norm{P_k}^2}=\begin{cases}
(k+1)\mu, & \mbox{for Charlier},\\[10pt]
\dst\frac{(k+1)(k+\gamma)\mu}{(1-\mu)^2}, & \mbox{for Meixner},
\end{cases}
\end{equation*}
while
\begin{equation*}
\frac{\kappa_{k+1-j}^2}{\kappa_{k-j}^2}=\begin{cases}
1, & \mbox{for Charlier},\\[10pt]
\dst(1-\mu)^{-2}, & \mbox{for Meixner}.
\end{cases}
\end{equation*}
Collecting these,
\begin{equation*}\label{eq:tratio}
\frac{t_{k+1}}{t_k}
=\frac{\kappa_{k+1-j}^2}{\kappa_{k-j}^2}\,(k-j-\alpha)^2\,
\frac{\norm{P_k}^2}{\norm{P_{k+1}}^2}\bigl(1+o(1)\bigr)
=\begin{cases}
\dfrac{(k-j-\alpha)^2}{(k+1)\mu}\bigl(1+o(1)\bigr),&\text{Charlier},\\[10pt]
\dfrac{(k-j-\alpha)^2}{(k+1)(k+\gamma)\mu}\bigl(1+o(1)\bigr),&\text{Meixner}.
\end{cases}
\end{equation*}
In the Meixner case the factor $(1-\mu)^{-2}$ from the scale cancels the
$(1-\mu)^{-2}$ in the norm ratio. Consequently
\begin{equation}\label{eq:tratiolimit}
\lim_{k\to\infty}\frac{t_{k+1}}{t_k}
=\begin{cases}
+\infty,&\text{Charlier},\\[10pt]
\dst\frac{1}{\mu},&\text{Meixner}.
\end{cases}
\end{equation}
We now draw from~\eqref{eq:tratiolimit} the geometric growth of the summands. Denote
by $L\in(1,+\infty]$ the limit in~\eqref{eq:tratiolimit}, which exceeds $1$ in both
families: it is $+\infty$ for Charlier, and equals $\mu^{-1}>1$ for Meixner because
$0<\mu<1$. Choose any $q$ with $1<q<L$, for instance $q=2$ in the Charlier case and
\begin{equation*}
    q=\frac{1+\dst\frac{1}{\mu}}{2},
\end{equation*}
in the Meixner case. By the definition of the limit, applied with
the threshold $q$, there exists $k_0\ge j$ such that
\begin{equation}\label{eq:ratiobound}
\frac{t_{k+1}}{t_k}\ge q,\quad k\ge k_0 .
\end{equation}
Iterating~\eqref{eq:ratiobound}, that is, multiplying the inequalities for
$k=k_0,k_0+1,\dots,m-1$ and telescoping the left-hand side, we obtain
\begin{equation}\label{eq:tgeom}
t_{m}\ge t_{k_0}\,q^{\,m-k_0},\quad m\ge k_0 .
\end{equation}
This is the precise sense in which the terms increase at least geometrically: they
dominate the geometric sequence $t_{k_0}q^{-k_0}\cdot q^{m}$ of ratio $q>1$. Three
consequences follow at once. First, since $t_{k_0}>0$ and $q>1$, the right-hand side
of~\eqref{eq:tgeom} tends to $+\infty$, hence $t_m\to\infty$. Second, the terms of
the series~\eqref{eq:Kjjdescent} therefore do not tend to $0$, so the series
diverges. Third, and quantitatively, summing~\eqref{eq:tgeom} over $m$ gives
\begin{equation*}
\Kn^{(j,j)}_{n-1}(\alpha,\alpha)=\sum_{k=j}^{n-1}t_k
\ge\sum_{k=k_0}^{n-1}t_{k_0}\,q^{\,k-k_0}
=t_{k_0}\,\frac{q^{\,n-k_0}-1}{q-1}\longrightarrow\infty ,
\end{equation*}
so that $\Kn^{(j,j)}_{n-1}(\alpha,\alpha)\to\infty$, as asserted.

It remains to sharpen this into the geometric lower bound~\eqref{eq:kernelgrowth}.
Since the kernel is a sum of positive terms, in particular
\begin{equation}\label{eq:Kgeqtop}
\Kn^{(j,j)}_{n-1}(\alpha,\alpha)\ge t_{n-1},
\end{equation}
and~\eqref{eq:tgeom} with $m=n-1$ gives $t_{n-1}\ge c_0\,q^{n}$, where
\begin{equation*}
c_0=t_{k_0}\,q^{-k_0-1}>0 .
\end{equation*}
Fix any $q'\in(1,q)$. Since $c_0\,(\frac{q}{q'})^{n}\to\infty$, there is $n_0\ge k_0$ with
$c_0\,q^{n}\ge(q')^{n}$ for all $n\ge n_0$; combined with~\eqref{eq:Kgeqtop}, this yields
\begin{equation*}
\Kn^{(j,j)}_{n-1}(\alpha,\alpha)\ge t_{n-1}\ge(q')^{n},\quad n\ge n_0,
\end{equation*}
which is~\eqref{eq:kernelgrowth} with $q'$ in the role of $q$. Finally, fix $p>0$;
since $q'>1$ one has
\begin{equation*}
    \frac{(q')^{n}}{n^{p}}\to\infty,
\end{equation*}
so
$n^{-p}\Kn^{(j,j)}_{n-1}(\alpha,\alpha)\to\infty$.
\end{proof}

\begin{lemma}\label{lem:Sndecomp}
Let $\{\Sob{n}\}_{n\ge0}$ be the Sobolev-type family of~\eqref{eq:innerprod} with
fixed $j\ge1$, $\alpha\in\R_-$ and $\lambda>0$. For $j\le k\le n-1$ define the
positive weights and confluent brackets
\begin{equation}\label{eq:Qdef}
Q_{n,k}(z)=P_n(z)-\dif^{j}P_n(\alpha)R_k(z),\quad R_k(z)=\frac{P_k(z)}{\dif^{j}P_k(\alpha)}.
\end{equation}
Then, for every $z\in\C$ and every $n>j$,
\begin{equation}\label{eq:Snexact}
\Sob{n}(z)
=\delta_n^{-1}\left(P_n(z)
+\lambda
\sum_{k=j}^{n-1}t_k\,Q_{n,k}(z)\right),
\end{equation}
and consequently, uniformly on compact subsets of $\C$,
\begin{equation}\label{eq:weightedavg}
\frac{\kappa_n^{-1}n}{\Gamma(n-z)}\,\Sob{n}(z)
=\frac{\displaystyle\sum_{i=0}^{n-1-j}\tau_{n,i}\,
\frac{\kappa_n^{-1}n}{\Gamma(n-z)}\,Q_{n,\,n-1-i}(z)}
{\displaystyle\sum_{i=0}^{n-1-j}\tau_{n,i}}\,\bigl(1+o(1)\bigr),
\quad
\tau_{n,i}=\frac{t_{n-1-i}}{t_{n-1}}.
\end{equation}
\end{lemma}

\begin{proof}
Since $\alpha<0$ lies outside the convex hull of the support of the orthogonality
measure, whose interior contains the zeros of $\{P_k\}$, one has
$\dif^{j}P_k(\alpha)\neq0$ for every $k\ge j$; hence the weights and
brackets~\eqref{eq:Qdef} are well defined and $t_k>0$.

We first establish the exact identity~\eqref{eq:Snexact}. Using the definitions of
the kernels~\eqref{eq:Kij} and~\eqref{eq:Kjjdescent}, we obtain the exact identity
\begin{eqnarray*}
P_n(z)\,\Kn^{(j,j)}_{n-1}(\alpha,\alpha)-\dif^{j}P_n(\alpha)\,\Kn^{(0,j)}_{n-1}(z,\alpha)
&=&\sum_{k=j}^{n-1}\frac{\left(\dif^{j}P_k(\alpha)\right)^2}{\norm{P_k}^2}
\Bigl(P_n(z)-\dif^{j}P_n(\alpha)R_k(z)\Bigr)\nonumber\\
&=&\sum_{k=j}^{n-1}t_k\,Q_{n,k}(z).\label{eq:combination}
\end{eqnarray*}
From which it follows that
\begin{eqnarray*}
P_n(z)(1+\lambda\,\Kn^{(j,j)}_{n-1}(\alpha,\alpha))
-\lambda\,\dif^{j}P_n(\alpha)\,\Kn^{(0,j)}_{n-1}(z,\alpha)
=P_n(z)+\lambda\sum_{k=j}^{n-1}t_k\,Q_{n,k}(z).
\end{eqnarray*}
Dividing the above identity by $\delta_n$ and
invoking~\eqref{eq:SobCon}, we deduce~\eqref{eq:Snexact}, valid for every $z\in\C$
and $n>j$.

We now derive~\eqref{eq:weightedavg}. Fix a compact set
$K\subset\C\setminus[0,+\infty)$ and multiply~\eqref{eq:Snexact} by
\begin{equation}\label{eq:factorMH}
    \frac{\kappa_n^{-1}n}{\Gamma(n-z)},
\end{equation}
we treat the two resulting terms separately. For the
first, the Mehler--Heine limit~\eqref{eq:mhcommon} gives
\begin{equation*}
    \frac{P_n(z)}{\kappa_n\Gamma(n-z)}\to\varphi(z), 
\end{equation*}
uniformly on $K$. Then, from
Lemma~\ref{lem:kernelgrowth}, we have
\begin{equation}\label{eq:firstterm}
\frac{n}{1+\lambda\Kn^{(j,j)}_{n-1}(\alpha,\alpha)}\cdot
\frac{P_n(z)}{\kappa_n\Gamma(n-z)}
=\varphi(z)\,\frac{n}{1+\lambda\Kn^{(j,j)}_{n-1}(\alpha,\alpha)}\bigl(1+o(1)\bigr)
\longrightarrow0,
\end{equation}
uniformly on $K$, since $\varphi$ is bounded on $K$.

For the second term, since Lemma~\ref{lem:kernelgrowth} gives
$\Kn^{(j,j)}_{n-1}(\alpha,\alpha)\to\infty$, the prefactor satisfies
\begin{equation}\label{eq:prefactor}
\frac{\lambda\,\Kn^{(j,j)}_{n-1}(\alpha,\alpha)}
{1+\lambda\,\Kn^{(j,j)}_{n-1}(\alpha,\alpha)}
=\frac{1}{1+\bigl(\lambda\,\Kn^{(j,j)}_{n-1}(\alpha,\alpha)\bigr)^{-1}}
=1+o(1),\quad n\to\infty,
\end{equation}
On the other hand, multiplying and dividing the second term
of~\eqref{eq:Snexact} by $\Kn^{(j,j)}_{n-1}(\alpha,\alpha)$, we obtain
\begin{equation}\label{eq:secondterm}
\frac{\lambda}{1+\lambda\Kn^{(j,j)}_{n-1}(\alpha,\alpha)}\sum_{k=j}^{n-1}t_k\,Q_{n,k}(z)
=\frac{\lambda\,\Kn^{(j,j)}_{n-1}(\alpha,\alpha)}{1+\lambda\Kn^{(j,j)}_{n-1}(\alpha,\alpha)}
\cdot\frac{\displaystyle\sum_{k=j}^{n-1}t_k\,Q_{n,k}(z)}{\displaystyle\sum_{k=j}^{n-1}t_k},
\end{equation}
whose first factor is $1+o(1)$ by~\eqref{eq:prefactor} and whose second factor is a
weighted average of the brackets $Q_{n,k}$ with the positive weights $t_k$.

It remains to reindex this weighted average from the top of the sum, where the
weights are largest. Setting $i=n-1-k$, so that $k=n-1-i$ runs over
$0\le i\le n-1-j$ as $k$ runs from $j$ to $n-1$, and dividing numerator and
denominator by the top weight $t_{n-1}>0$, the normalised weights
$\tau_{n,i}$ appear, with $\tau_{n,0}=1$, and
\begin{equation*}
\frac{\displaystyle\sum_{k=j}^{n-1}t_k\,Q_{n,k}(z)}{\displaystyle\sum_{k=j}^{n-1}t_k}
=\frac{\displaystyle\sum_{i=0}^{n-1-j}t_{n-1-i}\,Q_{n,\,n-1-i}(z)}
{\displaystyle\sum_{i=0}^{n-1-j}t_{n-1-i}}
=\frac{\displaystyle\sum_{i=0}^{n-1-j}\tau_{n,i}\,Q_{n,\,n-1-i}(z)}
{\displaystyle\sum_{i=0}^{n-1-j}\tau_{n,i}},
\end{equation*}
the division by $t_{n-1}$ leaving the quotient unchanged. Multiplying
by \eqref{eq:factorMH}, which enters each numerator term through
$Q_{n,\,n-1-i}$, and combining with~\eqref{eq:secondterm}, we arrive
\begin{equation*}
\frac{\kappa_n^{-1}n}{\Gamma(n-z)}\cdot
\frac{\lambda}{1+\lambda\Kn^{(j,j)}_{n-1}(\alpha,\alpha)}\sum_{k=j}^{n-1}t_k\,Q_{n,k}(z)
=\frac{\displaystyle\sum_{i=0}^{n-1-j}\tau_{n,i}\,
\frac{\kappa_n^{-1}n}{\Gamma(n-z)}\,Q_{n,\,n-1-i}(z)}
{\displaystyle\sum_{i=0}^{n-1-j}\tau_{n,i}}\,\bigl(1+o(1)\bigr).
\end{equation*}
Since the first term of~\eqref{eq:Snexact}, rescaled, tends to $0$
by~\eqref{eq:firstterm}, it is absorbed into the error, and~\eqref{eq:weightedavg}
follows uniformly on compact subsets of $\C$.
\end{proof}

\begin{lemma}\label{lem:prefratio}
Let $\alpha\in\R_-$ and $j\ge1$. For each fixed integer $l\ge0$, setting $m=n-l$,
\begin{equation}\label{eq:prefratio}
\frac{\dif^{j}P_n(\alpha)}{\dif^{j}P_m(\alpha)}
=\Bigl(-\frac{1}{c_\mu}\Bigr)^{\!l}n^{\,l}\bigl(1+o(1)\bigr),
\quad n\to\infty.
\end{equation}
\end{lemma}

\begin{proof}
For $l=0$ both sides equal $1$ and there is nothing to prove, so assume $l\ge1$.
By the forward shift operator~\eqref{eq:fwdgeneral},
$\dif^{j}P_s(\alpha)=[s]_j\,\widetilde P_{s-j}(\alpha)$ for every $s\ge j$, where
$\{\widetilde P_\cdot\}$ is the classical family descended $j$ times; this family
obeys~\eqref{eq:mhcommon},
and its limit function does not vanish at $\alpha$ since $\alpha<0$ avoids the
nonnegative integers. Writing the quotient as a telescoping product over the $l$
unit steps from $m=n-l$ to $n$,
\begin{equation}\label{eq:preftelescope}
\frac{\dif^{j}P_n(\alpha)}{\dif^{j}P_{n-l}(\alpha)}
=\prod_{r=1}^{l}\frac{\dif^{j}P_{n-r+1}(\alpha)}{\dif^{j}P_{n-r}(\alpha)},
\end{equation}
a product of $l$ consecutive-index ratios, each of which we now evaluate.

Fix $r\in\{1,\dots,l\}$ and set $s=n-r+1$. Applying the forward shift operator to
numerator and denominator, we deduce
\begin{equation}\label{eq:onestep}
\frac{\dif^{j}P_{s}(\alpha)}{\dif^{j}P_{s-1}(\alpha)}
=\frac{[s]_j}{[s-1]_j}\cdot
\frac{\widetilde P_{s-j}(\alpha)}{\widetilde P_{s-1-j}(\alpha)}
=\frac{s}{s-j}\cdot
\frac{\widetilde P_{s-j}(\alpha)}{\widetilde P_{s-1-j}(\alpha)}.
\end{equation}
The consecutive-index ratio~\eqref{eq:mhratios}, applied to
$\{\widetilde P_\cdot\}$ at the index $s-j$ and the point $\alpha$, gives
\begin{equation*}
\frac{\widetilde P_{s-j}(\alpha)}{\widetilde P_{s-1-j}(\alpha)}
=-\frac{s-1-j-\alpha}{c_\mu}\bigl(1+O(n^{-1})\bigr).
\end{equation*}
Substituting into~\eqref{eq:onestep} and using $s(s-j)^{-1}\to1$ together with
$(s-1-j-\alpha)(s-1-\alpha)^{-1}\to1$, both uniformly for $r\le l$ with $l$ fixed, we infer
\begin{equation}\label{eq:onesteplim}
\frac{\dif^{j}P_{s}(\alpha)}{\dif^{j}P_{s-1}(\alpha)}
=-\frac{s-1-\alpha}{c_\mu}\bigl(1+o(1)\bigr)
=-\frac{n-r-\alpha}{c_\mu}\bigl(1+o(1)\bigr),
\end{equation}
the last equality on recalling $s-1=n-r$.

Multiplying the $l$ relations~\eqref{eq:onesteplim} for $r=1,\dots,l$
in~\eqref{eq:preftelescope}, the $l$ factors $-c_\mu^{-1}$ combine into
$(-c_\mu^{-1})^{l}$ and the $l$ error factors $1+o(1)$ into a single $1+o(1)$, since
$l$ is fixed. This yields the first equality in~\eqref{eq:prefratio}. For the
second, each factor of the product satisfies $(n-r-\alpha)\,n^{-1}\to1$ as $n\to\infty$
with $r\le l$ fixed, so
\begin{equation*}
\prod_{r=1}^{l}(n-r-\alpha)=n^{\,l}\prod_{r=1}^{l}\frac{n-r-\alpha}{n}
=n^{\,l}\bigl(1+o(1)\bigr),
\end{equation*}
the product of finitely many factors each tending to $1$. Substituting gives the
second equality in~\eqref{eq:prefratio}.
\end{proof}

\begin{lemma}\label{lem:Qlimit}
Let $\alpha\in\R_-$, $j\ge1$, and let $Q_{n,k}$ and $R_k$ be as in~\eqref{eq:Qdef}.
Then, for each fixed integer $i\ge0$, we have
\begin{equation}\label{eq:Qlimit}
\frac{\kappa_n^{-1}n}{\Gamma(n-z)}\,Q_{n,\,n-1-i}(z)
\xrightarrow[n\to\infty]{}
-(i+1)(z-\alpha)\,\varphi(z),
\end{equation}
uniformly on compact subsets $K\subset\C$.
\end{lemma}

\begin{proof}
From~\eqref{eq:Qdef} the bracket reads $Q_{n,m}(z)=P_n(z)-\dif^{j}P_n(\alpha)R_m(z)$,
so its dependence on the lower index $m$ is carried entirely by $R_m$. Hence the
difference of two brackets with consecutive lower indices is
\begin{equation}\label{eq:Qdiff}
Q_{n,m}(z)-Q_{n,m+1}(z)
=\dif^{j}P_n(\alpha)\bigl(R_{m+1}(z)-R_m(z)\bigr),
\end{equation}
the terms $P_n(z)$ cancelling. Summing~\eqref{eq:Qdiff} over $m$ from $n-1-i$ to
$n-1$, the left-hand side telescopes to $Q_{n,\,n-1-i}(z)-Q_{n,n}(z)$; since
$Q_{n,n}(z)=P_n(z)-\dif^{j}P_n(\alpha)R_n(z)=P_n(z)-P_n(z)=0$, we obtain
\begin{equation}\label{eq:Qtelescope}
Q_{n,\,n-1-i}(z)
=\dif^{j}P_n(\alpha)\sum_{l=0}^{i}\bigl(R_{n-l}(z)-R_{n-1-l}(z)\bigr),
\end{equation}
a sum of exactly $i+1$ terms, indexed by $l=0,1,\dots,i$.

We next show that each summand of~\eqref{eq:Qtelescope} is a delay-zero bracket.
Fix $l$ and set $m:=n-l$. Placing the two fractions of $R_m-R_{m-1}$ over the common
denominator $\dif^{j}P_m(\alpha)\dif^{j}P_{m-1}(\alpha)$, we get
\begin{eqnarray}
\dif^{j}P_n(\alpha)\bigl(R_{m}(z)-R_{m-1}(z)\bigr)
&=&\dif^{j}P_n(\alpha)\Bigl(\frac{P_m(z)}{\dif^{j}P_m(\alpha)}
-\frac{P_{m-1}(z)}{\dif^{j}P_{m-1}(\alpha)}\Bigr)\nonumber\\
&=&\frac{\dif^{j}P_n(\alpha)}{\dif^{j}P_m(\alpha)}\cdot
\frac{\dif^{j}P_{m-1}(\alpha)\,P_m(z)-\dif^{j}P_m(\alpha)\,P_{m-1}(z)}
{\dif^{j}P_{m-1}(\alpha)}\nonumber\\
&=&\frac{\dif^{j}P_n(\alpha)}{\dif^{j}P_m(\alpha)}\cdot
\frac{\Nf_m(z)}{\dif^{j}P_{m-1}(\alpha)},
\label{eq:termasN}
\end{eqnarray}
where $\Nf_m(z)$ is precisely the bracket~\eqref{eq:Ndef} at top index $m=n-l$.

It remains to rescale~\eqref{eq:termasN} by \eqref{eq:factorMH} and let
$n\to\infty$ with $l$ fixed, so that $m=n-l\to\infty$ as well. The bracket at index
$m$ is governed by Lemma~\ref{lem:bracket}, which gives, uniformly on $K$, the following
\begin{equation}\label{eq:Nmlimit}
\frac{\Nf_m(z)}{\dif^{j}P_{m-1}(\alpha)}
=-\,\frac{z-\alpha}{m}\,\kappa_m\,\Gamma(m-z)\,\varphi(z)\,\bigl(1+o(1)\bigr),
\end{equation}
while the prefactor is supplied by Lemma~\ref{lem:prefratio}, which gives
\begin{equation}\label{eq:prefinserted}
\frac{\dif^{j}P_n(\alpha)}{\dif^{j}P_m(\alpha)}
=\Bigl(-\frac{1}{c_\mu}\Bigr)^{\!l}n^{\,l}\bigl(1+o(1)\bigr).
\end{equation}
The rescaling factor itself transfers from the index $n$ to the index $m$ through
the scale and Gamma quotients, which are obtained from
\eqref{eq:mhratios} iterated $l$ times and from the functional
equation of $\Gamma$,
\begin{equation}\label{eq:scalegamma}
\frac{\kappa_m}{\kappa_n}=(-c_\mu)^{\,l}\bigl(1+o(1)\bigr),
\quad
\frac{\Gamma(m-z)}{\Gamma(n-z)}=\prod_{r=1}^{l}\frac{1}{n-r-z}
=n^{-l}\bigl(1+o(1)\bigr).
\end{equation}
Inserting~\eqref{eq:Nmlimit}--\eqref{eq:scalegamma} into~\eqref{eq:termasN} and
using $m=n-l\sim n$, we deduce
\begin{equation}\label{eq:termlimit}
\frac{\kappa_n^{-1}n}{\Gamma(n-z)}\cdot
\frac{\dif^{j}P_n(\alpha)}{\dif^{j}P_m(\alpha)}\cdot
\frac{\Nf_m(z)}{\dif^{j}P_{m-1}(\alpha)}
\xrightarrow[n\to\infty]{}
-(z-\alpha)\,\varphi(z),
\end{equation}
uniformly on $K$. Indeed, the four factors carrying $l$-dependence, namely
$(-c_\mu^{-1})^{l}$ and $n^{l}$ from~\eqref{eq:prefinserted}, together with
$(-c_\mu)^{l}$ and $n^{-l}$ from~\eqref{eq:scalegamma}, cancel in pairs:
$(-c_\mu^{-1})^{l}(-c_\mu)^{l}=1$ and $n^{l}n^{-l}=1$. What survives is the
$l$-independent value $-(z-\alpha)\varphi(z)$, so every one of the $i+1$ summands
in~\eqref{eq:Qtelescope} contributes the same limit.

Summing the $i+1$ equal contributions~\eqref{eq:termlimit} over $l=0,1,\dots,i$, a
finite sum whose terms converge uniformly on $K$, yields~\eqref{eq:Qlimit}.
\end{proof}

\begin{lemma}[Asymptotics and summability of the weights]\label{lem:weights}
Let $t_k$ be as in~\eqref{eq:Kjjdescent} and $\tau_{n,i}$ as given by~\eqref{eq:weightedavg}.
Then, for each fixed $i\ge0$,
\begin{equation}\label{eq:taulim}
\tau_{n,i}\xrightarrow[n\to\infty]{}(1-c_\mu)^{\,i}.
\end{equation}
Consequently, the two series
\begin{equation}\label{eq:weightsums}
\sum_{i\ge0}(1-c_\mu)^{\,i}=\frac{1}{c_\mu},
\quad
\sum_{i\ge0}(i+1)(1-c_\mu)^{\,i}=\frac{1}{c_\mu^{2}},
\end{equation}
converge.
\end{lemma}

\begin{proof}
Clearly, taking reciprocals in~\eqref{eq:tratiolimit}, we arrive at
\begin{equation}\label{eq:recratio}
\frac{t_{k}}{t_{k+1}}\xrightarrow[k\to\infty]{}1-c_\mu,
\end{equation}
the limit being $0$ for Charlier, where $c_\mu=1$, and $\mu$ for Meixner, where
$1-c_\mu=\mu$; in the Charlier case this is the statement that the reciprocal of a
sequence diverging to $+\infty$ tends to $0$.

To establish~\eqref{eq:taulim}, fix $i\ge0$ and telescope the quotient $\tau_{n,i}$
over the $i$ unit steps from the index $n-1-i$ up to $n-1$,
\begin{equation}\label{eq:tautelescope}
\tau_{n,i}=\frac{t_{n-1-i}}{t_{n-1}}
=\prod_{l=0}^{i-1}\frac{t_{n-2-l}}{t_{n-1-l}},
\end{equation}
each factor being of the form~\eqref{eq:recratio} at the index $k=n-2-l$. As
$n\to\infty$ with $i$ fixed, every one of these $i$ indices tends to infinity, so
each factor tends to $1-c_\mu$ by~\eqref{eq:recratio}; the product of finitely many
convergent factors converges to the product of the limits, which
gives~\eqref{eq:taulim}.

Finally, since $x=1-c_\mu$ satisfies $0\le x<1$ in both families, the geometric
series and its termwise derivative converge, and
\begin{equation*}
\sum_{i\ge0}x^{\,i}=\frac{1}{1-x},
\quad
\sum_{i\ge0}(i+1)x^{\,i}=\frac{d}{dx}\sum_{i\ge0}x^{\,i+1}
=\frac{1}{(1-x)^{2}},
\end{equation*}
the termwise differentiation being legitimate inside the disc of convergence.
Substituting $x=1-c_\mu$, so that $1-x=c_\mu$, gives~\eqref{eq:weightsums}.
\end{proof}

\begin{proposition}\label{prop:closedform}
Let $\{\Sob{n}\}_{n\ge0}$ be the Sobolev-type family of~\eqref{eq:innerprod} with
fixed $j\ge1$, $\alpha\in\R_-$, $\lambda>0$. Then,
\begin{equation}\label{eq:reductionlimit}
\lim_{n\to\infty}\frac{\kappa_n^{-1}\,n}{\Gamma(n-z)}\,\Sob{n}(z)
=-\,\frac{(z-\alpha)\,\varphi(z)}{c_\mu},
\end{equation}
uniformly on compact subsets of $\C$. In particular, the limit is independent of $\lambda$ and of $j$.
\end{proposition}

\begin{proof}
Fix a compact set $K\subset\C\setminus[0,+\infty)$. The restriction will be removed
at the end using Vitali's theorem. Consider first the numerator
of~\eqref{eq:weightedavg}. For each fixed $i$, its general term converges pointwise,
uniformly on $K$. Indeed, by Lemma~\ref{lem:Qlimit} the rescaled bracket tends to
$-(i+1)(z-\alpha)\varphi(z)$, while by~\eqref{eq:taulim} the weight tends to
$(1-c_\mu)^{\,i}$, so that
\begin{equation}\label{eq:generalterm}
\tau_{n,i}\,\frac{\kappa_n^{-1}n}{\Gamma(n-z)}\,Q_{n,\,n-1-i}(z)
\xrightarrow[n\to\infty]{}
-(i+1)(z-\alpha)\,\varphi(z)\,(1-c_\mu)^{\,i}.
\end{equation}
Thus, we deduce
\begin{equation}\label{eq:numlimit}
\sum_{i=0}^{n-1-j}\tau_{n,i}\,\frac{\kappa_n^{-1}n}{\Gamma(n-z)}\,Q_{n,\,n-1-i}(z)
\xrightarrow[n\to\infty]{}
-(z-\alpha)\,\varphi(z)\sum_{i\ge0}(i+1)(1-c_\mu)^{\,i}
=-\,\frac{(z-\alpha)\,\varphi(z)}{c_\mu^{2}},
\end{equation}
the last equality being the second series of~\eqref{eq:weightsums}.

Then, by~\eqref{eq:taulim} and the first series
of~\eqref{eq:weightsums},
\begin{equation}\label{eq:denlimit}
\sum_{i=0}^{n-1-j}\tau_{n,i}\xrightarrow[n\to\infty]{}
\sum_{i\ge0}(1-c_\mu)^{\,i}=\frac{1}{c_\mu}.
\end{equation}
Since this limit is a nonzero constant, the quotient in~\eqref{eq:weightedavg}
converges to the quotient of the limits~\eqref{eq:numlimit} and~\eqref{eq:denlimit};
hence, uniformly on $K$,
\begin{equation}\label{eq:limitonK}
\frac{\kappa_n^{-1}n}{\Gamma(n-z)}\,\Sob{n}(z)
\xrightarrow[n\to\infty]{}
\frac{-(z-\alpha)\varphi(z)/c_\mu^{2}}{1/c_\mu}
=-\,\frac{(z-\alpha)\,\varphi(z)}{c_\mu}.
\end{equation}
The mass $\lambda$ disappeared in the passage to~\eqref{eq:weightedavg}, and the
order $j$ enters only through the ratios \eqref{eq:prefinserted}
defining the brackets $Q_{n,k}$, whose limit in Lemma~\ref{lem:Qlimit} is
independent of $j$; the limit therefore depends on neither, as asserted.

It remains to remove the restriction $z\notin[0,+\infty)$. Define
\begin{equation*}
    F_n(z)=\frac{\kappa_n^{-1}\,n}{\Gamma(n-z)}\,\Sob{n}(z).
\end{equation*}
Since $\Gamma^{-1}$ is an entire function, the poles of the rescaling factor are cancelled,
and therefore each $F_n$ is entire.

Fix $p\in\mathbb{N}_0$. The circle
\begin{equation*}
    C_p=\{z\in\mathbb{C}:|z-p|=1/2\},
\end{equation*}
is contained in $\mathbb{C}\setminus[0,+\infty)$.
By~\eqref{eq:limitonK}, the sequence $\{F_n\}$ converges uniformly on $C_p$.
Hence there exists a constant $M_p>0$ such that
\begin{equation*}
    |F_n(z)|\le M_p,
\quad
z\in C_p,\quad n\ge1.
\end{equation*}
Since each $F_n$ is entire, the Maximum Modulus Principle implies that the same bound
holds throughout the closed disc
\begin{equation*}
    D_p=\{z\in\mathbb{C}:|z-p|\le1/2\},
\end{equation*}
that is,
\begin{equation*}
    |F_n(z)|\le M_p,
\quad
z\in D_p,\quad n\ge1.
\end{equation*}
Away from the nonnegative integers, local boundedness already follows from the uniform
convergence established in~\eqref{eq:limitonK}. Therefore, $\{F_n\}$ is locally bounded
on the whole complex plane. Since $F_n$ converges on
$\mathbb{C}\setminus[0,+\infty)$, which has accumulation points in $\mathbb{C}$,
Vitali's theorem~\cite{Conway78,Mozo24} implies that $\{F_n\}$ converges locally
uniformly on $\mathbb{C}$ to a holomorphic function. This limit necessarily coincides
with the entire function given by the right-hand side of~\eqref{eq:limitonK}. Hence
\eqref{eq:reductionlimit} holds uniformly on compact subsets of $\mathbb{C}$, which
completes the proof.
\end{proof}

The two theorems follow by inserting the family-specific data
$\varphi,\kappa_n,c_\mu$ into Proposition~\ref{prop:closedform}.

\begin{theorem}\label{thm:mhcharlier}
Let $\{\Sob{n}\}_{n\ge0}$ be the monic Charlier-Sobolev type orthogonal polynomials
of~\eqref{eq:innerprod}, with $\mu>0$, $\alpha\in\R_{-}$, $\lambda>0$, $j\ge1$.
Then, uniformly on compact subsets of $\C$,
\begin{equation}\label{eq:mhScharlier}
\lim_{n\to\infty}\frac{(-1)^{n}\,n}{\Gamma(n-z)}\,\Sob{n}(z)
=-(z-\alpha)\,\frac{e^{\mu}}{\Gamma(-z)}.
\end{equation}
The limit is independent of $\lambda$ and $j$.
\end{theorem}

\begin{proof}
Here $\varphi=\varphi_C$, $\kappa_n=(-1)^n$ and $c_\mu=1$; substituting
into~\eqref{eq:reductionlimit} gives~\eqref{eq:mhScharlier}.
\end{proof}

\begin{theorem}\label{thm:mhmeixner}
Let $\{\Sob{n}\}_{n\ge0}$ be the monic Meixner-Sobolev type orthogonal polynomials
of~\eqref{eq:innerprod}, with $\gamma>0$, $0<\mu<1$, $\alpha\in\R_{-}$,
$\lambda>0$, $j\ge1$. Then, uniformly on compact subsets of $\C$,
\begin{equation}\label{eq:mhSmeixner}
\lim_{n\to\infty}\frac{(\mu-1)^{n}\,n}{\Gamma(n-z)}\,\Sob{n}(z)
=-\,\frac{z-\alpha}{(1-\mu)^{\gamma+z+1}\,\Gamma(-z)}.
\end{equation}
The limit is independent of $\lambda$ and $j$.
\end{theorem}

\begin{proof}
Here $\varphi=\varphi_M$, $\kappa_n=(\mu-1)^{-n}$ and $c_\mu=1-\mu$; substituting
into~\eqref{eq:reductionlimit}, the factor $1/(1-\mu)$ produces
$-(z-\alpha)\varphi_M(z)/(1-\mu)$, which is~\eqref{eq:mhSmeixner}.
\end{proof}

\subsection{Numerical illustration of the Mehler--Heine formulas}
\label{subsec:numericalMH}

Figures~\ref{fig:charlier-mehler-heine} and
\ref{fig:meixner-mehler-heine} provide a numerical illustration of the
Mehler--Heine formulas established in the previous section for the
discrete Sobolev-type Charlier and Meixner families. Rather than
plotting the Sobolev polynomials $\Sob{n}(z)$ themselves, the figures
display the corresponding normalised sequences appearing on the
left-hand sides of the Mehler--Heine limits. This normalisation removes
the dominant growth of the monic polynomials and allows the limiting
functions to be observed directly.

For the Charlier--Sobolev family, the plotted sequence is
\begin{equation*}
\mathcal{F}^{C}_{n}(z)
=
\frac{(-1)^n n}{\Gamma(n-z)}
\Sob{n}(z),
\label{eq:scaledCharlierFigure}
\end{equation*}
which, according to Theorem~\ref{thm:mhcharlier}, satisfies
\begin{equation*}
\mathcal{F}^{C}_{n}(z)
\longrightarrow
\mathcal{F}^{C}(z)
=
-\frac{(z-\alpha)e^{\mu}}
{\Gamma(-z)},
\quad n\to\infty,
\label{eq:limitCharlierFigure}
\end{equation*}
uniformly on compact subsets of $\mathbb{C}$.

Figure~\ref{fig:charlier-mehler-heine} compares the normalised
Charlier--Sobolev polynomials for
$n=50,100,150,200,250$ with the limiting function
$\mathcal{F}^{C}(z)$. The computations were performed using $\alpha=-3$, $\mu=5$, $j=10$ and $\lambda=100$. As the degree increases, the graphs converge rapidly towards the
limiting curve, illustrating the locally uniform convergence predicted
by the Mehler--Heine theorem. The factor $z-\alpha$ introduces an
additional simple zero at the exterior mass point $\alpha=-3$, whereas
the remaining zeros coincide with those of the reciprocal Gamma
function $1/\Gamma(-z)$. Consequently, the figure clearly exhibits the
two principal features of the asymptotic zero distribution: one
exceptional zero converges to the exterior mass point, while the
remaining zeros preserve the classical Charlier asymptotic behaviour.

For the Meixner--Sobolev family, the corresponding normalised sequence
is
\begin{equation*}
\mathcal{F}^{M}_{n}(z)
=
\frac{(\mu-1)^n n}
{\Gamma(n-z)}
\Sob{n}(z),
\label{eq:scaledMeixnerFigure}
\end{equation*}
and Theorem~\ref{thm:mhmeixner} yields
\begin{equation*}
\mathcal{F}^{M}_{n}(z)
\longrightarrow
\mathcal{F}^{M}(z)
=
-\frac{z-\alpha}
{(1-\mu)^{\gamma+z+1}\Gamma(-z)},
\quad n\to\infty,
\label{eq:limitMeixnerFigure}
\end{equation*}
uniformly on compact subsets of $\mathbb{C}$.

Figure~\ref{fig:meixner-mehler-heine} presents the analogous comparison
for the Meixner--Sobolev family using $\alpha=-3$, $\gamma=2$, $\mu=1/2$, $j=10$ and $\lambda=100$. Again, the normalised Sobolev polynomials converge rapidly towards the
predicted limiting function. Although the Meixner limit differs from
the Charlier one by the nonvanishing factor
$(1-\mu)^{-\gamma-z-1}$, both families exhibit the same qualitative
behaviour: the exterior Sobolev perturbation generates exactly one new
asymptotic zero located at the mass point $\alpha$, while all remaining
zeros converge to the classical locations determined by
$1/\Gamma(-z)$.

Taken together, Figures~\ref{fig:charlier-mehler-heine} and
\ref{fig:meixner-mehler-heine} provide numerical confirmation of the
main asymptotic results established in this paper. They show not only
the convergence of the scaled Sobolev polynomials towards their
Mehler--Heine limits, but also the universality phenomenon proved
theoretically: although the finite-degree polynomials depend on the
Sobolev parameters $\lambda$ and $j$, these quantities disappear from
the limiting functions. The only parameter that remains is the location
of the exterior mass point $\alpha$, whose effect is precisely the
appearance of the additional asymptotic zero.

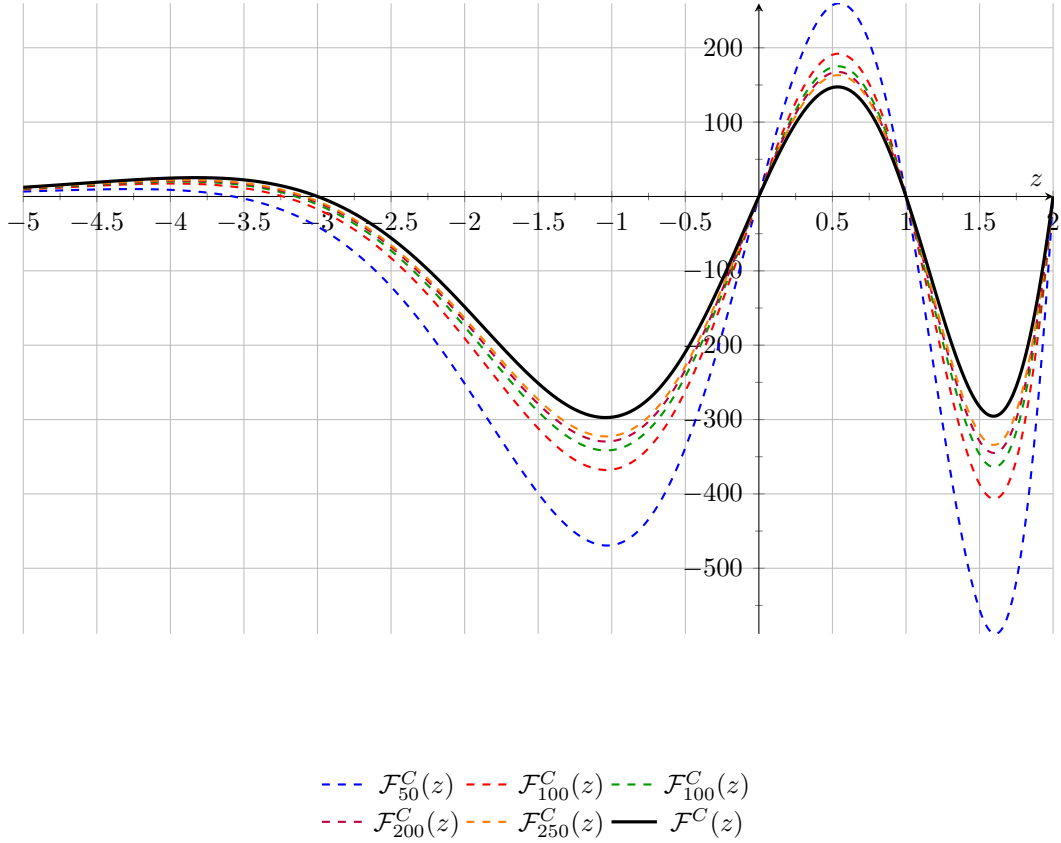
\begin{figure}[ht]
    \centering

    \begin{tikzpicture}
        \begin{axis}[
            width=0.95\textwidth,
            height=0.62\textwidth,
            xmin=-5,
            xmax=2,
            xlabel={$z$},
            ylabel={},
            axis lines=middle,
            grid=major,
            minor tick num=1,
            legend style={
                at={(0.5,-0.20)},
                anchor=north,
                legend columns=3,
                draw=none,
                font=\small
            },
            tick label style={
                font=\small
            },
            label style={
                font=\small
            },
            scaled ticks=false,
            unbounded coords=jump
        ]

        \addplot[
            blue,
            thick,
            dashed
        ]
        table[
            x expr=\thisrowno{0},
            y index=1
        ]{fChar_n50.dat};
        \addlegendentry{$\mathcal{F}^{C}_{50}(z)$}

        \addplot[
            red,
            thick,
            dashed
        ]
        table[
            x expr=\thisrowno{0},
            y index=1
        ]{fChar_n100.dat};
        \addlegendentry{$\mathcal{F}^{C}_{100}(z)$}

        \addplot[
            green!60!black,
            thick,
            dashed
        ]
        table[
            x expr=\thisrowno{0},
            y index=1
        ]{fChar_n150.dat};
        \addlegendentry{$\mathcal{F}^{C}_{100}(z)$}

        \addplot[
            purple,
            thick,
            dashed
        ]
        table[
            x expr=\thisrowno{0},
            y index=1
        ]{fChar_n200.dat};
        \addlegendentry{$\mathcal{F}^{C}_{200}(z)$}

        \addplot[
            orange,
            thick,
            dashed
        ]
        table[
            x expr=\thisrowno{0},
            y index=1
        ]{fChar_n250.dat};
        \addlegendentry{$\mathcal{F}^{C}_{250}(z)$}

        \addplot[
            black,
            very thick
        ]
        table[
            x expr=\thisrowno{0},
            y index=1
        ]{gChar_limit.dat};
        \addlegendentry{
            $\mathcal{F}^{C}(z)$
        }

        \end{axis}
    \end{tikzpicture}

    \caption{Convergence of the scaled Charlier--Sobolev polynomials
    toward the corresponding Mehler--Heine limit for
    $\alpha=-3$, $\mu=5$, $j=10$, and $\lambda=100$.}

    \label{fig:charlier-mehler-heine}
\end{figure}

\begin{figure}[ht]
    \centering

    \begin{tikzpicture}
        \begin{axis}[
            width=0.95\textwidth,
            height=0.62\textwidth,
            xmin=-3.5,
            xmax=2,
            xlabel={$z$},
            ylabel={},
            axis lines=middle,
            grid=major,
            minor tick num=1,
            legend style={
                at={(0.5,-0.20)},
                anchor=north,
                legend columns=3,
                draw=none,
                font=\small
            },
            tick label style={font=\small},
            label style={font=\small},
            scaled ticks=false,
            unbounded coords=jump
        ]

        \addplot[
            blue,
            thick,
            dashed
        ]
        table[
            x expr=\thisrowno{0},
            y index=1
        ]{fMeix_n50.dat};
        \addlegendentry{$\mathcal{F}^{M}_{50}(z)$}

        \addplot[
            red,
            thick,
            dashed
        ]
        table[
            x expr=\thisrowno{0},
            y index=1
        ]{fMeix_n100.dat};
        \addlegendentry{$\mathcal{F}^{M}_{100}(z)$}

        \addplot[
            green!60!black,
            thick,
            dashed
        ]
        table[
            x expr=\thisrowno{0},
            y index=1
        ]{fMeix_n150.dat};
        \addlegendentry{$\mathcal{F}^{M}_{150}(z)$}

        \addplot[
            purple,
            thick,
            dashed
        ]
        table[
            x expr=\thisrowno{0},
            y index=1
        ]{fMeix_n200.dat};
        \addlegendentry{$\mathcal{F}^{M}_{200}(z)$}

        \addplot[
            orange,
            thick,
            dashed
        ]
        table[
            x expr=\thisrowno{0},
            y index=1
        ]{fMeix_n250.dat};
        \addlegendentry{$\mathcal{F}^{M}_{250}(z)$}

        \addplot[
            black,
            very thick
        ]
        table[
            x expr=\thisrowno{0},
            y index=1
        ]{gMeix_limit.dat};
        \addlegendentry{$\mathcal{F}^{M}(z)$}

        \end{axis}
    \end{tikzpicture}

    \caption{Convergence of the scaled Meixner--Sobolev polynomials toward the corresponding Mehler--Heine limit for $\alpha=-3$, $\gamma=2$, $\mu=1/2$, $j=10$, and $\lambda=100$.}

    \label{fig:meixner-mehler-heine}
\end{figure}

\section{Conclusions}
\label{sec:conclu}

In this paper we have developed a unified analytic framework for the
discrete Sobolev-type Charlier and Meixner orthogonal polynomials
associated with arbitrary-order forward differences evaluated at an
exterior mass point. Starting from the connection formulas relating the
Sobolev-type and classical families, we established a systematic
generating-function approach that extends one of the most fundamental
analytical tools of the classical theory to a considerably broader
Sobolev setting.

The first main contribution of the paper is the derivation of explicit
generating functions for both Sobolev-type polynomial families together
with generating functions for their iterated forward differences. To the
best of our knowledge, these are the first generating-function
representations obtained for discrete Sobolev-type Charlier and Meixner
polynomials associated with arbitrary-order forward differences and an
exterior mass point. Their construction is nontrivial because the
Sobolev correction introduces the degree-dependent denominator $\delta_n$, which destroys the recursive mechanism underlying the classical
generating functions. The resulting representations therefore provide a
new analytical tool for investigating discrete Sobolev orthogonal
polynomials.

A second contribution is the derivation of new Mehler--Heine formulas
for both polynomial families. Unlike previously known results, which
were restricted to boundary masses, the present analysis considers an
exterior Sobolev mass and shows that the perturbation modifies the
leading asymptotic behaviour through the appearance of the factor
$(z-\alpha)$ in the limiting functions. Consequently, the asymptotic
limits preserve the classical Gamma-function structure while
incorporating explicitly the location of the exterior mass.

The Mehler--Heine formulas also provide a complete description of the
asymptotic distribution of the zeros. We proved that exactly one zero of
the Sobolev-type polynomials converges to the exterior mass point,
whereas all remaining zeros converge to the zeros of the corresponding
classical limiting functions. This establishes the precise asymptotic
effect of an exterior Sobolev perturbation and extends the known theory
beyond the previously studied boundary case.

Perhaps the most remarkable outcome of the present work is the
universality phenomenon exhibited by the limiting functions. Although
the Sobolev-type polynomials depend explicitly on the mass parameter
$\lambda$ and on the order $j$ of the forward difference operator, both
parameters disappear completely from the Mehler--Heine limits. The only
remaining trace of the Sobolev perturbation is the factor $(z-\alpha)$,
which records the location of the exterior mass through the appearance
of one additional asymptotic zero. This result reveals an unexpected
stability of the asymptotic regime under higher-order Sobolev
perturbations.

The numerical experiments presented in this paper fully support the
theoretical analysis. The scaled Sobolev-type polynomials exhibit rapid
convergence towards the predicted limiting functions, while the
evolution of their zeros agrees with the asymptotic behaviour derived
from the Mehler--Heine formulas.

Beyond the specific Charlier and Meixner families, the methodology
developed here suggests a general strategy for studying generating
functions and asymptotic properties of discrete Sobolev orthogonal
polynomials. In particular, the techniques introduced in this work may
be extended to other families of the Askey scheme and to more general
Sobolev perturbations involving different difference operators or
multiple mass points. We expect that the generating-function framework
developed in this paper will provide a useful foundation for future
investigations in these directions.





\section*{Acknowledgements}

The authors would like to thank the Department of Quantitative Methods at Universidad Loyola Andalusia for providing an excellent research environment and institutional support during the development of this work.

\section*{Declarations}

\subsection*{Ethical Approval}
Not applicable.

\subsection*{Consent to Participate}
Not applicable.

\subsection*{Consent to Publish}
Not applicable.

\subsection*{Data Availability Statement}
Not applicable.

\subsection*{Author Contributions}
Conceptualization, A.S.-L.; methodology, A.S.-L., J.-M., and E.J.-H.; software, A.S.-L. and J.-M.; validation, A.S.-L. and E.J.-H.; formal analysis, A.S.-L. and J.-M.; investigation, A.S.-L., E.J.-H., and J.-M.; resources, A.S.-L.; data curation, A.S.-L. and J.-M.; writing--original draft preparation, A.S.-L.; writing--review and editing, J.-M. and E.J.-H.; visualization, A.S.-L.; supervision, A.S.-L.; project administration, A.S.-L. All authors have read and agreed to the published version of the manuscript.

\subsection*{Funding}
This research received no external funding.

\subsection*{Competing Interests}
The authors declare that they have no competing interests.



\end{document}